\documentclass[10pt,reqno]{amsart}
\usepackage{latexsym,amsmath,mathtools}
\usepackage{amsfonts}
\usepackage{amssymb}
\usepackage{latexsym}
\usepackage{enumerate}
\usepackage[T1]{fontenc}
\usepackage{fourier}

\usepackage{bbm}

\usepackage{color}
\usepackage{fullpage}

\numberwithin{equation}{section}

\newcommand\KL[2]{D_{\mathrm{KL}}\bc{{\left.{#1}\right\|{#2}}}}

\newcommand\trans{^*}

\newcommand\vd{\vec d}
\newcommand\vn{\vec n}
\newcommand\hvn{\hat\vn}
\newcommand\vm{\vec m}
\newcommand\hvm{\hat\vm}
\newcommand\vnu{\vec\nu}
\newcommand\vmu{\vec\mu}
\newcommand{\vN}{\vec N}
\newcommand{\vM}{\vec M}

\newcommand\nix{\,\cdot\,}

\newcommand{\edgesethat}{\bar E(\hat\G)}
\newcommand{\eKathrin}{\edgesethat}
\newcommand{\auxH}{\bar E_{1}(\hat\G)}
\newcommand{\Sdir}{\bar{G}(S)}
\newcommand{\Sundir}{G(S)}
\newcommand{\stir}{\mathcal S}

\newcommand\sign{\mathrm{sign}}

\newcommand\G{\vec G}

\newcommand\br[1]{\left(#1\right)}

\def\vec#1{\mathchoice{\mbox{\boldmath$\displaystyle#1$}}
{\mbox{\boldmath$\textstyle#1$}}
{\mbox{\boldmath$\scriptstyle#1$}}
{\mbox{\boldmath$\scriptscriptstyle#1$}}}

\DeclareMathOperator{\pr}{\mathbb P}

\newtheorem{definition}{Definition}[section]
\newtheorem{claim}[definition]{Claim}

\newtheorem{theorem}[definition]{Theorem}
\newtheorem{lemma}[definition]{Lemma}
\newtheorem{proposition}[definition]{Proposition}
\newtheorem{corollary}[definition]{Corollary}

\newtheorem{fact}[definition]{Fact}

\newcommand\cA{\mathcal{A}}
\newcommand\cB{\mathcal{B}}

\newcommand\cD{\mathcal{D}}
\newcommand\cF{\mathcal{F}}

\newcommand\cE{\mathcal{E}}

\newcommand\cN{\mathcal{N}}
\newcommand\cQ{\mathcal{Q}}

\newcommand\cM{\mathcal{M}}

\newcommand\cV{\mathcal{V}}

\def\cE{{\mathcal E}}

\newcommand\eul{\mathrm{e}}
\newcommand\eps{\varepsilon}

\newcommand\NN{\mathbb{N}}

\newcommand\Var{\mathrm{Var}}

\newcommand\Erw{\mathbb{E}}

\newcommand{\vecone}{\vec{1}}

\newcommand{\Po}{{\rm Po}}
\newcommand{\Bin}{{\rm Bin}}

\newcommand{\diag}{\operatorname{diag}}

\newcommand{\bink}[2] {{{#1}\choose {#2}}}

\newcommand\bc[1]{\left({#1}\right)}
\newcommand\cbc[1]{\left\{{#1}\right\}}
\newcommand\bcfr[2]{\bc{\frac{#1}{#2}}}
\newcommand{\bck}[1]{\left\langle{#1}\right\rangle}
\newcommand\brk[1]{\left\lbrack{#1}\right\rbrack}
\newcommand\scal[2]{\bck{{#1},{#2}}}

\newcommand{\whp}{w.h.p.}

\newcommand{\stacksign}[2]{{\stackrel{\mbox{\scriptsize #1}}{#2}}}

\newcommand{\Erdos}{Erd\H{o}s}
\newcommand{\Renyi}{R\'enyi}

\newcommand{\Bollobas}{Bollob\'as}

\newcommand{\Luczak}{\L uczak}

\newcommand\Lem{Lemma}
\newcommand\Prop{Proposition}
\newcommand\Thm{Theorem}
\newcommand\Cor{Corollary}
\newcommand\Sec{Section}

\newcommand\RSA{Random Structures and Algorithms}
\newcommand\JCTB{Journal of Combinatorial Theory, Series~B}

\newcommand{\conds}{\mathcal{E}}
\newcommand{\setB}{\hat \cN_+}
\newcommand{\sizeB}{\hat n_+}

\begin{document}

\title{Core Forging and local limit theorems for the $k$-core of random graphs}

\author{Amin Coja-Oghlan$^*$, Oliver Cooley$^{**}$, Mihyun Kang$^{**}$ and Kathrin Skubch}
\thanks{$^*$The research leading to these results has received funding from the European Research Council under the European Union's Seventh 
Framework Programme (FP/2007-2013) / ERC Grant Agreement n.\ 278857--PTCC\\
$^{**}$Supported by Austrian Science Fund (FWF): P26826 and W1230, Doctoral Program ``Discrete Mathematics''.}

\address{Amin Coja-Oghlan, {\tt acoghlan@math.uni-frankfurt.de}, Goethe University, Mathematics Institute, 10 Robert Mayer St, Frankfurt 60325, Germany.}

\address{Oliver Cooley, {\tt cooley@math.tugraz.at}, Graz University of Technology, Institute of Discrete Mathematics, Steyrergasse 30, 8010 Graz, Austria}

\address{Mihyun Kang, {\tt kang@math.tugraz.at}, Graz University of Technology, Institute of Discrete Mathematics, Steyrergasse 30, 8010 Graz, Austria}

\address{Kathrin Skubch, {\tt skubch@math.uni-frankfurt.de}, Goethe University, Mathematics Institute, 10 Robert Mayer St, Frankfurt 60325, Germany.}

\begin{abstract}
\noindent
We establish a multivariate local limit theorem for the order and size as well as several other parameters of the $k$-core of the \Erdos-\Renyi\ random graph.
The proof is based on a novel approach to the $k$-core problem that replaces the meticulous analysis of the `peeling process' by a generative model of graphs with a core of a given order and size.
The generative model, which is inspired by the Warning Propagation message passing algorithm, facilitates the direct study of properties of the core and its connections with the mantle and should therefore be of interest in its own right.

\end{abstract}

\maketitle

\section{Introduction}\label{Sec_intro}

\subsection{The $k$-core problem}
\noindent
The {\em $k$-core} of a graph $G$ is the largest subgraph of minimum degree at least $k$.
It can be determined algorithmically by the {\em peeling process} that removes an arbitrary vertex of degree less than $k$ while there is one.
In one of the most influential contributions to the theory of random graphs Pittel, Spencer and Wormald
analysed the peeling process on the \Erdos-\Renyi\ random graph via the method of differential equations~\cite{Pittel}.
They determined the precise threshold $d_k$ from where the $k$-core is non-empty \whp\
as well as the asymptotic order (number of vertices) and size (edges) of the $k$-core for $d>d_k$, $k\geq3$.
The case $k\geq3$ is very different from the case $k=2$, as the $2$-core simply emerges continuously along with the giant component.
By contrast, 
a most remarkable feature of the case  $k\geq3$, first observed by \Luczak~\cite{Tomasz1,Tomasz2},
is that the order of the $k$-core leaps from $0$ to a linear number of vertices at the very moment that the $k$-core becomes non-empty.

Since the seminal work of Pittel, Spencer and Wormald several alternative 
derivations of the $k$-core threshold have been put forward
~\cite{Cooper, Fernholz1, Fernholz2, JansonLuczak, Kim, MolloyCores,Riordan,Sato}.
Some of these extend to hypergraphs and/or given degree sequences.
Additionally, establishing a bivariate central limit theorem, Janson and Luczak~\cite{JansonLuczak2} studied the joint limiting distribution of the order and size of the $k$-core.
Further aspects of the problems that have been studied include the `depth' of the peeling process as well as the width of the
critical window~\cite{DMcore,Gao,GaoMolloy}.

The great interest in the $k$-core problem is  due not least to the many connections that the  problem has with
other questions in combinatorics and computer science.
For example, coinciding with the largest $k$-connected subgraph \whp,
the $k$-core problem is a natural generalisation of the `giant component' problem~\cite{BB}.
Cores also play a very important role in the study of {random constraint satisfaction problems}
such as random $k$-SAT or random graph colouring.
In these problems the emergence of a core-like structure causes freezing, a particular kind of long-range correlations that has been associated with the algorithmic difficulty of finding solutions~\cite{Barriers, Molloy}.
In addition, the hypergraph version of the $k$-core holds the key to understanding problems such as random XORSAT, hypergraph orientability and cuckoo hashing~\cite{Dietzfelbinger,Nikos,PittelSorkin}.
The problem plays an important role in coding theory as well~\cite{Luby}.

While most of the previous work on the $k$-core problem has been based on tracing the peeling process,
	the only exception being~\cite{Riordan}, reliant on branching processes,
in the present paper we develop a very different approach.
We devise a generative model for
random graphs with a $k$-core of a given order and size.
Formally, we develop a randomised sampling algorithm {\tt Forge} that produces a graph with a core of a given desired order and size 
(under certain reasonable assumptions on the input parameters).
The output distribution of {\tt Forge} converges in total variation to the distribution of an \Erdos-\Renyi\ random graph
given the order and size of the $k$-core.
Because the randomised construction employed by {\tt Forge} is surprisingly simple, we can immediately read off
the asymptotic number of graphs with a $k$-core of a given order and size.
As an application, we obtain a bivariate {\em local limit theorem} for the distribution of the order and size of the $k$-core of the \Erdos-\Renyi\ random graph.
This result substantially sharpens the central limit theorem of Janson and Luczak~\cite{JansonLuczak2}.
Additionally, the sampling algorithm completely elucidates the way the $k$-core is embedded into the random graph, a question on which we obtained
partial results in an earlier paper via the formalism of local weak convergence~\cite{localcore}.
We expect that this structural insight will facilitate the future study of the $k$-core and of similar structures arising in random constraint satisfaction problems.

The paper  is almost entirely self-contained and most of the proofs are elementary.
The only (mildly) advanced ingredient that we use is a local limit theorem for sums of independent random variables~\cite{McD}.
In particular, we do not rely on any of the previous results on the $k$-core, not even the one on the location of the $k$-core threshold.

\subsection{A  local limit theorem}\label{Sec_LLT}
Let $\G=\G(n,m)$ be the random graph with $n$ vertices and $m=\lceil dn/2 \rceil$ edges, where
$d>0$ is independent of $n$.
Moreover, for an integer $k\geq3$ consider the function
	\begin{equation}\label{eqmain}
	\phi_{d,k}:[0,1]\to[0,1],\qquad x\mapsto\pr\brk{\Po(dx)\geq k-1}=1-\exp(-dx)\sum_{j=0}^{k-2}\frac{(dx)^j}{j!}\enspace.
	\end{equation}
Clearly, $\phi_{d,k}$ is continuous and $\phi_{d,k}(0)=0$. 
Let $p=p(d,k)\in[0,1]$ be the largest fixed point of $\phi_{d,k}$ and set
	\begin{equation}\label{eqdk}
	d_k=\inf\{d>0:p(d,k)>0\}.
	\end{equation}
In addition, define
	\begin{align}\label{eq:q}
	q&=q(d,k)=\pr\brk{\Po(dp)=k-1|\Po(dp)\geq k-1}=\frac{d^{k-1}p^{k-2}\exp(-dp)}{(k-1)!}.
	\end{align}		

\begin{theorem}\label{Cor_LLT}
Suppose that $k\geq3$, $d>d_k$ and fix any number $\xi>0$.
Then $1-(k-1)q>0$ and the $2\times 2$ matrix
	\begin{align}\label{eqCor_LLT}
	\cQ&={(1-(k-1)q)}^{-2}\begin{pmatrix}\cQ_{11}&\cQ_{12}\\\cQ_{21} &\cQ_{22}\end{pmatrix}
	\end{align}
with
	\begin{align*}
\cQ_{11}&=
-{\left(d k^{2} - 2 \, d k + d\right)} p^{2} q^{4} - {\left(2 \, {\left(d^{2} k - d^{2}\right)} p^{3} - {\left(2 \, d k^{2} - d^{2} + {\left(d^{2} - 2 \, d\right)} k\right)} p^{2} + {\left(d k^{2} - 2 \, d k + d\right)} p\right)} q^{3} - d p^{2}
 \nonumber\\
 &\quad
- {\left({\left(d^{3} + 2 \, d^{2}\right)} p^{4} - {\left(d^{3} + 2 \, d^{2} k\right)} p^{3} + {\left(d k^{2} - d^{2} + 2 \, {\left(d^{2} + d\right)} k - 2 \, d\right)} p^{2} - {\left(d k^{2} - d\right)} p\right)} q^{2} + d p
\nonumber
 \\
 &\quad
- {\left(2 \, d^{2} p^{3} - 2 \, {\left(d^{2} + d k\right)} p^{2} + {\left(2 \, d k - d\right)} p\right)} q,
\nonumber\\
\cQ_{12}&=\cQ_{21}=
-2 \, d p^{3} + 2 \, d p^{2} - 2 \, {\left({\left(d k - d\right)} p^{4} + {\left(d k - d\right)} p^{3}\right)} q^{2} - 2 \, {\left({\left(d^{2} + d\right)} p^{4} - {\left(d^{2} + d k\right)} p^{3} + {\left(d k - d\right)} p^{2}\right)} q,
\nonumber\\
\cQ_{22}&=
-4 \, {\left(k - 1\right)} p^{4} q - 2 \, {\left(2 \, d + 1\right)} p^{4} + 4 \, d p^{3} - 2 \, {\left({\left(k^{2} - 2 \, k + 1\right)} p^{4} + {\left(k^{2} - 2 \, k + 1\right)} p^{2}\right)} q^{2} + 2 \, p^{2}\nonumber
	\end{align*}
is regular.
Further, let $X$ be the order of the $k$-core of $\G$ and let $Y$ be its size.
Then uniformly for all integers $x,y$ such that
$|x-np(1-q)|+|y-mp^2|\leq\xi\sqrt n$ we have
	\begin{align}\label{eqLLT}
	 \pr\brk{X=x,Y=y}
		& \sim\frac{\sqrt{\det\cQ}}{\pi d n} \exp\br{ -\frac n2  \scal{\cQ\bink{x/n-p(1-q)}{y/m-p^2}} {\bink{x/n-p(1-q)}{y/m-p^2}}}.
	\end{align}
\end{theorem}

The formula (\ref{eqLLT}) determines the asymptotic probability that the order and size $X,Y$ of the $k$-core attain specific values
within $O(\sqrt n)$ of their expectations.
Hence, \Thm~\ref{Cor_LLT} provides a bivariate {\em local limit theorem} for the order and size of the $k$-core.
This result is significantly stronger than a mere {\em central limit theorem} stating that
$X,Y$ converge jointly to a bivariate Gaussian because (\ref{eqLLT}) actually yields the asymptotic point probabilities.
Still it is worthwhile pointing out that \Thm~\ref{Cor_LLT} immediately implies a central limit theorem.

\begin{corollary}\label{Cor_CLT}
Suppose that $k\geq3$ and $d>d_k$, let $\cQ$ be the matrix from (\ref{eqCor_LLT}) and let
$X,Y$ be the order and size of the $k$-core of $\G$.
	Then $n^{-1/2}((X-np(1-q)),2(Y-mp^2)/d)$
converges in distribution to a bivariate Gaussian with mean $0$ and covariance matrix $\cQ^{-1}$.
\end{corollary}

A statement similar to \Cor~\ref{Cor_CLT} was previously established by  Janson and Luczak \cite{JansonLuczak2} via a careful analysis of the peeling process.
However, they did not obtain an explicit formula for the covariance matrix.
Indeed, although the formula for $\cQ$ is a bit on the lengthy side, the only 
non-algebraic quantity is $p=p(d,k)$, the solution to the fixed point equation.
By contrast, the formula of Janson and Luczak implicitly characterises the covariance matrix in terms of another stochastic process, and
they do not provide a {\em local} limit theorem.

The number $d_k$ from (\ref{eqdk}) does, of course, coincide with the $k$-core threshold first derived in~\cite{Pittel}.
The formula given in that paper looks a bit different but we pointed out the equivalence in~\cite{localcore}.
In fact, it is {\em very} easy to show that the $k$-core is empty \whp\ if $d<d_k$.
On the other hand,  \Cor~\ref{Cor_CLT} implies immediately that for $d>d_k$ the $k$-core contains $n(p(1-q)+o(1))=\Omega(n)$ vertices \whp\
Since the proofs of \Thm~\ref{Cor_LLT} and \Cor~\ref{Cor_CLT} do not assume knowledge of the $k$-core threshold, we thus obtain a new derivation of the threshold result.

\subsection{Warning Propagation}\label{Sec_WP}
A key idea of the present paper is to investigate not merely the $k$-core itself but
 also the ``surrounding structure'' of the graph from the right angle.
As it turns out, the necessary additional parameters can be set out concisely by way of the Warning Propagation message passing algorithm
introduced in non-rigorous physics work on random constraint satisfaction problems~\cite{MM}.
The link between Warning Propagation and the $k$-core problem is well known~\cite{localcore,Ibrahimi,MM}.
The important feature that we highlight and exploit here is that the Warning Propagation messages allow us to describe succinctly how the $k$-core is embedded into the rest of the random graph, the {\em mantle}.
More precisely, as we pointed out in~\cite{localcore} Warning Propagation gives rise naturally to a few further parameters apart from the order and size of the $k$-core that are of combinatorial significance
but that, unfortunately, get lost in the peeling process.
The main result of the paper, \Thm~\ref{Thm_LLT} below, provides a local limit theorem for the joint distribution of all these parameters.

Warning Propagation assigns \emph{messages} to edges, one in either direction, and \emph{marks} to vertices.
The messages and the marks are $\{0,1\}$-valued.
Initially all messages are set to $1$.
Thus, for a graph $G=(V(G),E(G))$ we let $\mu_{v\to w}(0|G)=1$ for all pairs $(v,w)\in V(G)\times V(G)$ such that $\{v,w\}\in E(G)$.
Subsequently the messages get updated in parallel rounds. 
That is,  writing $\partial v=\partial_Gv$ for the neighbourhood of vertex $v$
and abbreviating $\partial v\setminus w
=\partial v\setminus\{w\}$, we inductively define
\begin{equation}\label{WP_messageupdate}
\textstyle \mu_{v\to w}(t+1|G)
 =\vecone\cbc{\sum_{u\in\partial v\setminus w}\mu_{u\to v}(t|G)\geq k-1}\qquad\mbox{for integers }t\geq0.
\end{equation}
We emphasise that the messages are directed and quite possibly $\mu_{v\to w}(t|G)\neq\mu_{w\to v}(t|G)$.
Additionally, the mark of $v\in [n]$ at time $t\geq0$ is defined as
\begin{equation}\label{WP_markupdate}
\textstyle \mu_v(t|G)=\vecone\cbc{\sum_{u\in\partial v}\mu_{u\to v}(t|G)\geq k}.
\end{equation}
Clearly, $\mu_{v\to w}(t+1|G)\leq\mu_{v\to w}(t|G)$ for all $t\geq0$ and all $v,w$.
Hence, $\mu_{v}(t+1|G)\leq\mu_v(t|G)$ for all $v$ and the limits
	$$\mu_v(G)=\lim_{t\to\infty}\mu_{v}(t|G),\qquad\mu_{v\to w}(G)=\lim_{t\to\infty}\mu_{v\to w}(t|G)$$
exist for all $v,w$. 
Denote by 
$$\mu(G)=\br{\mu_v(G),\mu_{v\to w}(G)}_{v\in V(G),\{v,w\} \in E(G)}.$$
The following observation is immediate from the construction.

\begin{fact}[{\cite[\Lem~3.1]{localcore}}]\label{fact_coreWP}
Let $G$ be a graph.
\begin{enumerate}
\item A vertex $u$ belongs to the $k$-core of $G$ iff $\mu_u(G)=1$.
\item An edge $\{v,w\}$ links two vertices of the $k$-core iff $\mu_{v\to w}(G)=\mu_{w\to v}(G)=1$.
\end{enumerate}
\end{fact}

The messages encode how the $k$-core is embedded into the mantle.
To see this, we introduce
	\begin{align*}
		\cN_0(G)&=\textstyle\cbc{v:\sum_{u\in\partial v}\mu_{u\to v}(G)
		\leq  k-2},\\
	\cN_\star(G)&=\textstyle\cbc{v:\sum_{u\in\partial v}\mu_{u\to v}(G)= k-1},\\
	\cN_1(G)&=\textstyle\cbc{v:\sum_{u\in\partial v}\mu_{u\to v}(G)\geq k},\\
	\cM_{xy}(G)&=\cbc{(v,w)\in V(G)^2:\{v,w\}\in E(G),\, \mu_{w\to v}(G)=x,\,\mu_{v\to w}(G)=y}
		\qquad\qquad\qquad\qquad(x,y\in\{0,1\}).
	\end{align*}
Fact~\ref{fact_coreWP} shows that $\cN_1(G)$ is just the vertex set of the $k$-core.
Moreover, the vertices in $\cN_\star(G)$ miss out on core membership by just one incoming $1$-message.
In effect, if they receive a $0$ message from a neighbour, they send back a $1$, and vice versa.
By contrast, the vertices in $\cN_0(G)$ send out $0$ messages to all their neighbours, although they may receive up to $k-2$ many $1$-messages.
Further, Fact~\ref{fact_coreWP} implies that $(v,w)\in\cM_{11}(v)$ iff the edge $\{v,w\}$ connects two vertices inside the $k$-core.
Similarly, if $(v,w)\in\cM_{10}(G)$, then $v\in\cN_\star(G)\cup\cN_0(G)$ and $w\in\cN_1(G)\cup\cN_\star(G)$, and
 $(v,w)\in\cM_{10}(G)$ iff $(w,v)\in\cM_{01}(G)$.
Finally, if $(v,w)\in\cM_{00}(G)$, then $v,w\in\cN_0(G)$.

Given this Warning Propagation-inspired decomposition of the vertices and edges, the
key parameters of the $k$-core problem are
	$$n_0(G)=|\cN_0(G)|,\qquad \ n_\star(G)=|\cN_\star(G)|,\qquad \ n_1(G)=|\cN_1(G)|,\qquad \ m_{xy}(G)=|\cM_{xy}(G)|.$$
Of course, by Fact~\ref{fact_coreWP} the order of the $k$-core equals $n_1(G)$ and its size is equal to $m_{11}(G)/2$.
Further, both $m_{00}(G)$ and $m_{11}(G)$ are even and
	\begin{align}\label{eqaffine}
	n_0(G)+n_1(G)+n_\star(G)&=|V(G)|,&
	m_{01}(G)&=m_{10}(G),&m_{00}(G)+m_{01}(G)+m_{10}(G)+m_{11}(G)&=2|E(G)|.
	\end{align}
In effect, the seven parameters
$$\vn(G)=(n_0(G),n_\star(G),n_1(G))\quad\mbox{and}\quad\vm(G)=(m_{00}(G),m_{01}(G),m_{10}(G),m_{11}(G))$$
boil down to the four variables $$\vN(G)=(n_\star(G),n_1(G))\quad\mbox{and}\quad\vM(G)=(m_{10}(G),m_{11}(G)).$$
Then we have the following local limit theorem for $\vN(\G),
\vM(\G)$.

\begin{figure}\small
	\begin{align*}
	Q_{11}
	&=-
	\frac 1d\left(
	 {\left(d k^{2} - 2 \, d k + d\right)} p^{2} q^{4} + {\left(2 \, {\left(d^{2} k - d^{2}\right)} p^{3} - {\left(2 \, d k^{2} - d^{2} + {\left(d^{2} - 2 \, d\right)} k\right)} p^{2} + {\left(d k^{2} - 2 \, d k + d\right)} p\right)} q^{3} 
\right.	 
\\
&\quad 
\left.
- d p q + \left({\left(d^{3} + 2 \, d^{2}\right)} p^{4} - {\left(d^{3} + 2 \, {\left(d^{2} + 2 \, d\right)} k - 4 \, d\right)} p^{3} + {\left({\left(d + 2\right)} k^{2} - d^{2} + 2 \, {\left(d^{2} - 2\right)} k + 2\right)} p^{2}	 
\right.
\right.
\\
&\quad 
\left.
\left.
 - {\left(d k^{2} - 2 \, d k + d\right)} p\right) q^{2}
\right),
\\
Q_{12}&= {\left(k^{2} - 2 \, k + 1\right)} p^{2} q^{4} + {\left(2 \, {\left(d k - d\right)} p^{3} - {\left({\left(d - 2\right)} k + 2 \, k^{2} - d\right)} p^{2} + {\left(k^{2} - 2 \, k + 1\right)} p\right)} q^{3} 
\\
&\quad 
+ {\left({\left(d^{2} + 2 \, d\right)} p^{4} - {\left(d^{2} + 2 \, {\left(d + 1\right)} k - 2\right)} p^{3} 
+ {\left({\left(2 \, d + 1\right)} k + k^{2} - d - 1\right)} p^{2} - {\left(k^{2} - k\right)} p\right)} q^{2} 
\\
&\quad 
+ {\left(d p^{3} - {\left(d + k\right)} p^{2} + {\left(k - 1\right)} p\right)} q,
\\
Q_{13}
&= -
\frac 1d\left(
{\left(2 \, {\left(d k - d\right)} p^{4} + 2 \, {\left({\left(d + 2\right)} k - k^{2} - d - 1\right)} p^{3} - 3 \, {\left(d k - d\right)} p^{2} + {\left({\left(d - 2\right)} k + k^{2} - d + 1\right)} p\right)} q^{2} 
\right.	 
\\
&\quad 
\left.
+ {\left(2 \, {\left(d^{2} + d\right)} p^{4} - {\left(3 \, d^{2} + 2 \, {\left(d + 1\right)} k + 2 \, d - 2\right)} p^{3} + {\left(d^{2} + {\left(3 \, d + 2\right)} k - 2\right)} p^{2} - {\left({\left(d + 1\right)} k - 1\right)} p\right)} q
\right),
\\
Q_{14}
&= \frac 2d\left({\left({\left(d k - d\right)} p^{4} + {\left({\left(d + 2\right)} k - k^{2} - d - 1\right)} p^{3}\right)} q^{2} + {\left({\left(d^{2} + d\right)} p^{4} - {\left(d^{2} + {\left(d + 1\right)} k - 1\right)} p^{3} + {\left(d k - d\right)} p^{2}\right)} q\right),
\\
Q_{22}&= -{\left(k^{2} - 2 \, k + 1\right)} p^{2} q^{4} - {\left(2 \, {\left(d k - d\right)} p^{3} - {\left({\left(d - 2\right)} k + 2 \, k^{2} - d\right)} p^{2} + {\left(k^{2} - 2 \, k + 1\right)} p\right)} q^{3}
\\
&\quad 
 - {\left({\left(d^{2} + 2 \, d\right)} p^{4} - {\left(d^{2} + 2 \, d k\right)} p^{3} + {\left(2 \, {\left(d + 1\right)} k + k^{2} - d - 2\right)} p^{2} - {\left(k^{2} - 1\right)} p\right)} q^{2} - p^{2} 
\\
&\quad  
 - {\left(2 \, d p^{3} - 2 \, {\left(d + k\right)} p^{2} + {\left(2 \, k - 1\right)} p\right)} q + p,
\\
Q_{23}
&= 2 \, p^{3} + {\left(2 \, {\left(k - 1\right)} p^{4} + 2 \, {\left(k - 1\right)} p^{3} - 3 \, {\left(k - 1\right)} p^{2} + {\left(k - 1\right)} p\right)} q^{2} - 3 \, p^{2} 
\\
&\quad 
+ {\left(2 \, {\left(d + 1\right)} p^{4} - {\left(3 \, d + 2 \, k + 2\right)} p^{3} + {\left(d + 3 \, k\right)} p^{2} - k p\right)} q + p,
\\
Q_{24}
&= -2 \, p^{3} - 2 \, {\left({\left(k - 1\right)} p^{4} + {\left(k - 1\right)} p^{3}\right)} q^{2} + 2 \, p^{2} - 2 \, {\left({\left(d + 1\right)} p^{4} - {\left(d + k\right)} p^{3} + {\left(k - 1\right)} p^{2}\right)} q,
\\
Q_{33}&= -
\frac 1d\left(
2 \, {\left(2 \, d + 1\right)} p^{4} - 4 \, {\left(2 \, d + 1\right)} p^{3} + {\left(5 \, d + 3\right)} p^{2} + {\left(2 \, {\left(k^{2} - 2 \, k + 1\right)} p^{4} - 2 \, {\left(k^{2} - 2 \, k + 1\right)} p^{2} + {\left(k^{2} - 2 \, k + 1\right)} p\right)} q^{2} 
\right.	 
\\
&\quad 
\left.
- {\left(d + 1\right)} p + {\left(4 \, {\left(k - 1\right)} p^{4} - 4 \, {\left(k - 1\right)} p^{3} + {\left(k - 1\right)} p^{2}\right)} q\right),
\\
Q_{34}&= 
\frac 2d\left(
{\left(k^{2} - 2 \, k + 1\right)} p^{4} q^{2} + {\left(2 \, d + 1\right)} p^{4} - {\left(3 \, d + 1\right)} p^{3} + d p^{2} + {\left(2 \, {\left(k - 1\right)} p^{4} - {\left(k - 1\right)} p^{3}\right)} q
\right),
\\
Q_{44}&= -
\frac 2d\left(
2 \, {\left(k - 1\right)} p^{4} q + {\left(2 \, d + 1\right)} p^{4} - 2 \, d p^{3} + {\left({\left(k^{2} - 2 \, k + 1\right)} p^{4} + {\left(k^{2} - 2 \, k + 1\right)} p^{2}\right)} q^{2} - p^{2}\right)
\end{align*}
	\caption{The matrix entries $Q_{ij}$.}\label{Fig_Sage}
\end{figure}

\begin{theorem}\label{Thm_LLT}
Suppose that $k\geq3$, $d>d_k$ and $\xi>0$.
Then the symmetric $4\times 4$-matrix
	\begin{align}\label{eqQmatrix}
	Q&=
	\frac 1{(1-(k-1)q)^2}
	\bc{Q_{ij}}_{1\leq i,j\leq 4}
	\end{align}
with $Q_{ij}$ from Figure~\ref{Fig_Sage}
is regular and uniformly for all integer vectors $\vN=(n_\star,n_1),\vM=(m_{10},m_{11})$ 
 such that $m_{11}$ is even and  	\begin{equation}\label{eqthm:contig2}
	|n_\star-n\nu_\star|+|n_1-n\nu_1|+|m_{10}-2m\mu_{10}|+|m_{11}-2m\mu_{11}|\leq\xi\sqrt n
	\end{equation}
 we have
	\begin{align}\nonumber
	\pr\brk{
	\vN(\G)=\vN,\vM(\G)=\vM}
    &=\frac{1}{2(\pi d n)^{2}\sqrt{\det Q}}
    \exp\br{-\frac n2
    \bck{Q^{-1}\Delta(\vN,\vM), \Delta(\vN,\vM)}} 
    + o(n^{-2})
    \end{align}
    where
    \begin{align}
    \Delta(\vN,\vM)&=
    \begin{pmatrix}
    \begin{array}{ccc}
    n_\star/n&-&pq,\\
    n_1/n&-&p(1-q)\\
    m_{10}/(2m)&-&p(1-p)
    \\m_{11}/(2m)&-&p^2
    \end{array}
    \end{pmatrix}
    .\label{eqDELTANM}
	\end{align}
\end{theorem}

\noindent
\Thm~\ref{Cor_LLT} is immediate from \Thm~\ref{Thm_LLT} by just projecting
on $n_{1}(\G)$ and $m_{11}(\G)/2$.

\subsection{Techniques, outline and further related work}
We do {\em not} prove \Thm~\ref{Thm_LLT} by analysing Warning Propagation on $\G$.
Instead, we are going to employ the seven parameters supplied by Warning Propagation in order to set up a generative process {\tt Forge} for creating a random graph with a core of a given order and size and, more specifically, with given values of $\vN,\vM$.
The proof of \Thm~\ref{Thm_LLT} is then  based on simply counting the number of graphs that {\tt Forge} can produce for given $\vN,\vM$.

In a prior paper~\cite{localcore} we used Warning Propagation to describe the {\em local} structure of the core, the mantle and the interactions between the two.
More specifically, take a random graph $\G$ with average degree $d>d_k$ and colour the vertices inside the core black and those outside white.
Then it is clear that each black vertex has at least $k$ black neighbours, while a white vertex has at most $k-1$ black neighbours.
But how are the white vertices interconnected?
Clearly a white vertex can easily have more than $k$ white neighbours.
Yet the connections between the white vertices are subject to seemingly complicated constraints.
An obvious one is that no two white vertices with precisely $k-1$ black neighbours can be adjacent.
Indeed, if we tried to get by with just the two ``types'' black and white then an infinite set of such constraints arises.
In~\cite{localcore} we showed that these local interactions can be described neatly and elegantly in terms of a 5-type branching process, with the types inspired by Warning Propagation, and established a corresponding local weak convergence result.
Thus, the offspring matrix of this 5-type branching process succinctly expresses the infinite set of constraints on the connections between the white vertices.
A similar result about local weak convergence was established in~\cite{Ibrahimi} for the $2$-core of random hypergraphs.
However, these methods do not suffice to obtain a global generative process such as {\tt Forge}.

Kim~\cite{Kim} provided a very simple generative model, the Poisson cloning model, of the {\em internal} structure of the $k$-core.
This model has been used to study properties of the $k$-core itself (see, for example, ~\cite{Nikos}).
The generative model behind {\tt Forge} can be seen as a substantial extension of the Poisson cloning model that
encompasses both the core and the mantle.
In effect, {\tt Forge} greatly facilitates the direct analysis of properties of the core, the mantle and the connections between them.
For example, it would be very easy to read results on the ``depth'' of the peeling process off the generative model.
We believe that this approach is much simpler than the direct analysis of the peeling process as performed, e.g., in \cite{AMXOR} for the hypergraph $2$-core, and that it will find future applications, e.g., in the theory of random constraint satisfaction problems.

In \Sec~\ref{Sec_Forge} we present Warning Propagation and the sampling algorithm {\tt Forge}.
In \Sec~\ref{Sec_strategy}
we outline the analysis of {\tt Forge} and the counting argument that yields the asymptotic number of graphs with a given outcome of $\vN,\vM$.
The details of that analysis follow in the remaining sections.

\subsection{Notation and preliminaries}\label{Sec_prelims}
With respect to general notation, we let $G[S]$ denote the subgraph of a graph $G=(V(G),E(G))$ induced on $S\subset V(G)$.
Moreover, the transpose of a matrix $A$ is denoted by $A^*$ and
 for reals  $a_1,\ldots,a_s$ we let  $\diag(a_1,\ldots,a_s)$ be the  $s\times s$ diagonal matrix with diagonal entries $a_1,\ldots,a_s$.

In addition to the parameters $p=p(d,k)$, which we defined as the largest fixed point of the function $\phi_{d,k}$ from (\ref{eqmain}),
and $q$ from (\ref{eq:q}) we introduce
	\begin{align}\label{eq:qbar}
	\bar q&=\bar q(d,k)=\pr\brk{\Po(dp)=k-2|\Po(dp)\leq k-2}.
	\end{align}		
The definitions of $p$ and $q$ ensure that
\begin{align}
\label{eq:qbarexp}
\bar q &=\frac{(k-1)q}{(1-p)d}.
\end{align}
\noindent 
Furthermore,
a bit of calculus reveals the following.

\begin{fact}[{\cite[Lemma 2.3.]{localcore}}]\label{fact_FPana}
Let $k\geq 3$ and $d>d_k$ and 
let $p$ be the largest fixed point of
$\phi_{d,k}$. Then 
\begin{enumerate}
 \item $p\geq\frac{k-2+\sqrt{k-2}}{d}$;
 \item $\frac{\partial}{\partial x}
 \left.\phi_{d,k}(x)\right|_{x=p}=q(k-1)=\bar q(1-p)d<1$.
\end{enumerate}
\end{fact}

Throughout the paper we will frequently encounter truncated Poisson distributions.
To be precise, for real numbers $y,z>0$ we let $\Po_{\geq z}(y)$ denote the Poisson distribution $\Po(y)$ conditioned on the event that the outcome is at least $z$.
Thus,
	$$\pr\brk{\Po_{\geq z}(y)=\ell}=\frac{\vecone\cbc{\ell\geq z}y^\ell\exp(-y)}{\ell!\pr\brk{\Po(y)\geq z}}\qquad\mbox{for any integer $\ell\geq0$}.$$
The distributions $\Po_{> z}(y)$, $\Po_{\leq z}(y)$, $\Po_{< z}(y)$ are defined analogously.
We will also occasionally encounter the function
	\begin{equation}\label{eq:varphi}
	\varphi_{\ell}:[0,1]\to[0,1],\qquad 
	y\mapsto\pr\brk{\Po(y)\geq \ell-1}\qquad\qquad(\ell\geq3),
	\end{equation}
whose derivatives work out to be
 	\begin{align}\label{eq:deriv}
 	\frac{\partial}{\partial y}\varphi_\ell(y)
 	&=\frac{y^{\ell-2}}{(\ell-2)!\exp(y)},\qquad
	\frac{\partial^2}{\partial y^2}\varphi_\ell(y)
	=\frac{(\ell-y-2)y^{\ell-3}}{(\ell-2)!\exp(y)}.
	\end{align}
In particular, recalling $\phi_{d,k}$ from \eqref{eqmain}, we see that
$\phi_{d,k}(x)=\varphi_k(dx)$ for 
all $x\in [0,1]$
and
\begin{align*}
\frac{\partial^i}{\partial x^i}\phi_{d,k}(x)
=d^i\frac{\partial^i}{\partial y^i}\left.\varphi_\ell(y)
\right|_{y=d\cdot x}\qquad\qquad(i\geq 0,\ k\geq 3).
\end{align*}

The following standard result shows that joint convergence to a family of independent Poisson variables can be established by way of calculating joint factorial moments.

\begin{theorem}[{\cite{BB}}]\label{Cor_mom}
Let $(X_n^{(i)})_{i\geq 1}$ be a family of random variables. If
$\lambda_i$, $i\geq 0$ are such that for all $r_1,\ldots,r_m\geq 0$, 
\begin{equation*}
\lim_{n\to\infty}\Erw\brk{(X_n^{(1)})_{r_1}\cdots (X_n^{(m)})_{r_m}}=\lambda_1^{r_1}\cdots\lambda_m^{r_m}, 
\end{equation*}
then $(X_n^{(i)})_{i\geq 1} \to (Z_i)_{i\geq 1}$ in distribution, where $Z_i$ are independent with distribution $\Po(\lambda_i)$.
\end{theorem}

\noindent
Furthermore, in Section \ref{section_LLTproofs} we will need the following local limit theorem for sums of independent random variables.

\begin{theorem}[{\cite[Theorem 2.1]{McD}}]\label{Thm_McD}
Let $\ell\geq 1$. For $n\geq 1$ let $X_{1,n},\ldots,X_{n,n}$ be a sequence of independent $\mathbb N^\ell$-valued random variables.
Let $\vecone_r\in\mathbb N^\ell$ denote the vector whose $r$-th component is $1$ 
and whose other components are $0$.
Assume that there is a constant $c>0$ such that 
for all $r\leq \ell$ and $n\geq 1$,
\begin{equation*}
\max_{\vec k\in \mathbb N^\ell}\min\left\{\pr\br{X_{i,n}=\vec k},\pr\br{X_{i,n}=\vec k+\vecone_r}\right\}\geq c.
\end{equation*}
Then for $S_n=\sum_{i=1}^nX_{i,n}$ the following holds.
Suppose that there is a 
vector $\vec a$ in $\mathbb R^\ell$ such that 
$n^{-1/2}(S_n-\vec a)$ converges in distribution to 
a multivariate normal distribution with mean $0$ and 
covariance matrix $D$. Then uniformly for all 
vectors $\vec k\in\mathbb N^\ell$,
\begin{align*}
\pr\br{S_n=\vec k}=\frac{1}{\sqrt{(2\pi n)^\ell\det D}}\exp\br{-\frac n2
\scal{D^{-1}\br{\frac{\vec k}n - \vec a}}{\br{\frac{\vec k}n - \vec a}}}+\mbox{o}\br{n^{-\ell/2}}.
\end{align*}
\end{theorem}

Additionally, we need a few basic combinatorial counting results.
We recall that for an integer $\ell$ the number of perfect matchings of the complete graph of order $2\ell$ is equal to
	\begin{align}\label{eq!!}
	(2\ell-1)!!=\frac{(2\ell)!}{2^\ell \ell!}\enspace.
	\end{align}
	
\noindent
Further,  for $s,t\in\NN$ let $\stir(s,t)$ denote the Stirling number of the second kind.

\begin{theorem}[{\cite[Theorem 3]{stirling_new}}]\label{thm_stir}
For all $s,t\in\NN$ we have
$
\stir(s,t)\leq \frac 12 t^{s-t} \bink{s}{t}.
$
\end{theorem}

\noindent
We need the following upper bound on the number of labelled forests that comes in terms of the Stirling number.

\begin{theorem}[{\cite[\Cor~3.1]{Peter}}]\label{thm_peter}
The number of labelled forests on $v$ vertices with exactly $\ell$ leaves and exactly $c$ components 
is upper bounded by
	$$\frac{v!}{\ell!}\bink{v-1}{c-1}\stir(v-c,v-\ell).$$
\end{theorem}

\noindent
The entropy of a probability distribution $\rho$ on a finite set $\Omega\neq\emptyset$ is defined as 
	\begin{align}\label{eqHDef}
	H(\rho)=-\sum_{\omega\in\Omega}\rho(\omega)\ln(\rho(\omega)).
     \end{align}
Further, we recall that for two probability distributions $\rho,\rho'$ on the  same finite set $\Omega\neq\emptyset$ the Kullback-Leibler 
divergence is defined as
	\begin{align}\label{eqKLDef}
	\KL{\rho}{\rho'}=\sum_{\omega\in\Omega}\rho(\omega)\ln{\frac{\rho(\omega)}{\rho'(\omega)}},
	\end{align} 
with the convention that $0\ln 0=0\ln\frac00=0$ and $\KL{\rho}{\rho'}=\infty$ if there is $\omega\in\Omega$ such that 
	$\rho(\omega)=0<\rho'(\omega)$.
The derivatives of a generic summand on the right hand side of (\ref{eqKLDef}) work out to be
\begin{align}\label{eqDiffKL}
 \frac{\partial}{\partial x} x\ln{\frac xy}
 &= 1+\ln{\frac{x}{y}} ,& 
 \frac{\partial^2}{\partial x^2} x\ln{\frac xy}
 &=\frac 1x.
\end{align}

\bigskip\noindent
{\bf\em From here on we tacitly assume that $k\geq3$ and $d>d_k$. We continue to use the notation from \Sec s~\ref{Sec_LLT} and~\ref{Sec_prelims} throughout the paper.}

\section{Core forging}\label{Sec_Forge}

\noindent
The key insight of the present paper is that the extra information provided by the Warning Propagation algorithm can easily be turned into a generative process for creating random graphs with a core of a given order and size (under certain reasonable assumptions).
To set up this generative process, we need a few further parameters: let
	\begin{align}
	\mu_{00}&=(1-p)^2,&\mu_{01}&=\mu_{10}=p(1-p),&\mu_{11}&=p^2,\nonumber\\
		\nu_0&=1-p,
	&\nu_\star&=pq,
	&\nu_1&=p(1-q)\nonumber\\
	\vnu&= (\nu_0,\nu_\star,\nu_1),&\vmu&=(\mu_{00},\mu_{01},\mu_{10},\mu_{11}).
	&	 \label{eq_bar}
	\end{align}
In light of \Thm~\ref{Thm_LLT} the (intended) semantics of $\vnu,\vmu$ is clear: $\nu_z$ is going to emerge as the expectation of $n_z(\G)/n$ for $z\in\{0,1,\star\}$ and $\mu_{yz}$ as that of $m_{yz}(\G)/(2m)$ for $y,z\in\{0,1\}$.

Further, let us write $d_G(v)$ for the degree of vertex $v$ in a graph $G$ and let $d_{G,ab}(v)$
be the number of vertices $w\in\partial_Gv$ such that $\mu_{w\to v}(G)=a$ and $\mu_{v\to w}(G)=b$ for $a,b\in\{0,1\}$.
Then it is immediate from the definitions (\ref{WP_messageupdate}), (\ref{WP_markupdate}) of the Warning Propagation marks and messages that the sets $\cN_0(G),\cN_\star(G),\cN_1(G)$ can be characterised in terms of the degrees $d_{G,ab}$ as follows.

\begin{fact}\label{Fact_WP}
Let $G$ be a graph.
\begin{enumerate}
\item $v\in\cN_0(G)$ iff $d_{G,10}(v)\leq k-2$ and $d_{G,11}(v)=d_{G,01}(v)=0$.
\item $v\in\cN_\star(G)$ iff $d_{G,10}(v)=k-1$ and $d_{G,11}(v)=d_{G,00}(v)=0$.
\item $v\in\cN_1(G)$ iff $d_{G,11}(v)\geq k$ and $d_{G,10}(v)=d_{G,00}(v)=0$.
\end{enumerate}
\end{fact}

\noindent
Finally, introducing
	\begin{align}\label{eqlambda}
	 \lambda_{00} =\lambda_{01} &= d(1-p),& \lambda_{10} = \lambda_{11} &= dp,
	 \end{align}
we will see that the parameters $\lambda_{ab}$ govern the distributions of the degrees $d_{\G,ab}(v)$, subject to the conditions listed in Fact~\ref{Fact_WP}.
	
We can now  describe the randomised algorithm {\tt Forge} that 
generates a graph $\hat\G$ along with a set of `supposed' Warning Propagation  messages $\hat\vmu$, see Figure~\ref{Fig_Forge}.
In the first step {\tt Forge} randomly assigns each vertex a type $0,\star,1$ independently according to the distribution $\vnu$.
The second step generates a sequence $(\hat d_{ab}(v))_{a,b,v}$ of `pseudo-degrees' by independently
sampling from the conditional Poisson distributions with parameters $\lambda_{ab}$.
Of course, in order to ultimately generate a graph with $m$ edges it had better be the case that the total degree sum come to $2m$, which step (3) checks.
In addition, we require that the total $00$ and $11$-degree sums be even and that $\hat m_{10}=\hat m_{01}$.
Hence, if $\hat m_{00},\hat m_{01},\hat m_{10},\hat m_{11}$ fail to satisfy any of the conditions from~(\ref{eqaffine}), then the algorithm aborts.
Since the $\hat m_{ab}$ are sums of independent random variables, we verify easily that the success probability of step (3) is $\Theta(n^{-1})$.

The next two steps of {\tt Forge} use the $(\hat d_{ab}(v))_{a,b,v}$ to generate a random graph from an enhanced version of
the configuration model of graphs with given degree distributions.
More precisely, for each vertex $v$ we create $\hat d_{ab}(v)$ {half-edges} of type $ab$ for every $a,b\in\{0,1\}$.
Then we create a random matching of the half-edges that respects the types.
That is, a half-edge of type $11$ has to be matched to another one of type $11$, a half-edge of type $00$ gets matched to another $00$ half-edge  and the $10$ half-edges get matched to the $01$ ones.
The conditions on $\hat m_{00},\ldots,\hat m_{11}$ from step (3) guarantee that such a matching exists.
We check right away whether the resulting graph $\hat\G$ is simple (i.e.\ contains no loops or multiple edges) and abort if it is not.

Step (6) sets up pseudo-messages $\hat\vmu_{v\to w}\in\{0,1\}$ for every pair $(v,w)$.
These reflect the intuition that guided  the construction of the graph.
That is, we set $\hat\vmu_{v\to w}$ to the value that we believe the actual Warning Propagation messages $\mu_{v\to w}(\hat\G)$ ought to take.
The final step of the algorithm checks whether the actual Warning Propagation on $\hat\G$ meet these expectations.
If $\hat\vmu_{v\to w}(\hat\G)\neq\mu_{v\to w}(\hat\G)$ for some vertex pair $v,w$, the algorithm aborts.
Otherwise it outputs $\hat\G$.

The following theorem shows that the success probability of {\tt Forge} is not too small and that given success the output distribution is close to the \Erdos-\Renyi\ random graph in total variation.

\begin{theorem}\label{thm:contig}
If $k\ge 3$ and $d>d_k$, then the success probability of\: {\tt Forge}$(n,m)$
is $\Omega(n^{-1})$ and 
	the total variation distance of $\G$ and $\hat\G$ given success is $o(1)$.
\end{theorem}

\begin{figure}
\small
{\bf Algorithm} {\tt Forge}$(n,m)$.\\
\raggedright
\begin{enumerate}
\item 
Partition the vertex set $[n]$ randomly into 
three sets $\hat\cN_0,\hat\cN_\star,\hat\cN_1$, with vertex $v$
being placed into set $\cN_x$ with probability $\nu_x$ for
$x\in\{0,\star,1\}$ independently.
Let $\hat n_0 =|\hat \cN_0|$, $\hat n_\star =|\hat \cN_\star|$,
		$\hat n_1=|\hat \cN_1|$ and 
		$\hat{\vn}=(\hat n_0,\hat n_\star, \hat n_1)$.

\item For each vertex $v$ independently let
	$$\chi_{00}(v)=\Po(\lambda_{00}),\quad
		\quad\chi_{01}(v)=\Po(\lambda_{01}),\quad
		\chi_{10}(v)=\Po_{\leq k-2}(\lambda_{10}),\quad
		\chi_{11}(v)=\Po_{\geq k}(\lambda_{11})$$
	and
	\begin{align*}
	\hat d_{00}(v)&=\chi_{00}(v)
	\vecone\{v\in\hat\cN_0\},&
	\hat d_{01}(v)&=\chi_{01}(v)
	\vecone\{v\in\hat\cN_\star\cup\hat\cN_1\},\\
	\hat d_{10}(v)&=(k-1)\vecone\{v\in\hat\cN_\star\}
	+\chi_{10}(v)\vecone\{v\in\hat\cN_0\},&
	\hat d_{11}(v)&=\chi_{11}(v)\vecone\{v\in\hat\cN_1\}.
	\end{align*}
Let 
$$
		\hat m_{00}=\sum_{v\in[n]}\hat d_{00}(v),\qquad
		\hat m_{01}=\sum_{v\in[n]}\hat d_{01}(v),\qquad
		\hat m_{10}=\sum_{v\in[n]}\hat d_{10}(v),\qquad
		\hat m_{11}=\sum_{v\in[n]}\hat d_{11}(v).
$$
and
	    $\hat{\vm} = (\hat m_{00},\hat m_{01},
	    \hat m_{10},\hat m_{11})$.
\item 
	If either  
	$\hat m_{00}$ or $\hat m_{11}$ are odd, 
	$\hat m_{01}\neq \hat m_{10}$
    or $\hat m_{00}+2\hat m_{01}+\hat m_{11}\neq2m$
    then output {\tt failure} and abort.
\item
	Else let 
		\begin{align*}
		V_{00}&=\bigcup_{v\in\hat\cN_0}\{(v,0,0)\}\times[\hat d_{00}(v)],&
		V_{01}&=\bigcup_{v\in\hat\cN_\star\cup\hat\cN_1}\{(v,0,1)\}\times[\hat d_{01}(v)],\\
		V_{10}&=\bigcup_{v\in\hat\cN_\star\cup\hat\cN_0}\{(v,1,0)\}\times[\hat d_{10}(v)],&
		V_{11}&=\bigcup_{v\in\hat\cN_1}\{(v,1,1)\}\times[\hat d_{11}(v)].
		\end{align*}
	Independently generate uniformly random perfect matchings 
	$\hat\cM_{00}$ of the complete graph $K_{V_{00}}$,
	$\hat\cM_{11}$ of $K_{V_{11}}$ and $\hat\cM_{10}$ 
	of the complete bipartite graph $K_{V_{01},V_{10}}$.
\item	Let $\hat\G$ be the multi-graph  obtained from 
	$\hat\cM_{00}\cup \hat\cM_{10}\cup \hat\cM_{11}$ by contracting the sets  $\{(v,x,y,z):x,y\in\{0,1\},z\in[d_{xy}(v)]\}$ to the single  vertex $v$.
	If $\hat\G$ fails to be simple, then output {\tt failure} and stop.
\item 
	Let $\hat{\vec\mu}_v=\vecone\{v\in\hat\cN_1\}$  for all $v\in[n]$.
	Moreover, for $(v,w)\in[n]\times[n]$ set
		\begin{align*}
		\hat{\vec\mu}_{v\to w}
		&=\vecone\{v\in\hat\cN_1,w\in\partial_{\hat\G}v\}
		+\vecone\{v\in\hat\cN_\star,
		\exists i,j:\{(v,0,1,i),(w,1,0,j)\}\in\hat \cM_{10}\}.
		\end{align*}
Let $E(\hat \G)$ be the edge set of $\hat\G$
and 
$$\hat\vmu =\br{\hat\vmu_v,
\hat\vmu_{v\to w}}_{v\in [n],\{v,w\}\in E(\hat\G)}.$$
\item  If
	  $\hat{\vec\mu}\neq \mu(\hat\G)$, then output {\tt failure}.
	  Otherwise output $\hat\G$ and declare {\tt success}.
\end{enumerate}
\caption{The algorithm {\tt Forge}.}\label{Fig_Forge}
\end{figure}

\Thm~\ref{thm:contig} makes it easy to analyse properties of the core of the \Erdos-\Renyi\ graph, the mantle and the connections between them.
Indeed, all we need to do is to investigate {\tt Forge}, which samples from a fairly accessible random graph model composed of nothing but independent random variables and random matchings.
There are ample techniques for studying such models.
In particular, \Thm~\ref{thm:contig} shows that any property that the pair $(\hat\G,\hat\vmu)$ enjoys with probability $1-o(1/n)$
holds for the pair $(\G,\mu(\G))$ \whp\
In fact, the $1/n$-factor in the success probability comes exclusively from the harmless conditioning in step~(3).
Thus, if $(\hat\G,\hat\vmu)$ has a property \whp\ given that step~(3) does not abort, then the same property holds for $(\G,\mu(\G))$ \whp\

We proceed to state an enhanced version of \Thm~\ref{thm:contig} that allows us to condition on the order and size of the $k$-core.
To this end,  given integer vectors $\vN=(n_\star,n_1)$ and $\vM=(m_{10},m_{11})$ such that $m_{11}$ is even
let $\cF(\vN,\vM)$ be the event that {\tt Forge} succeeds and 
$\hat n_\star=n_\star,\hat n_1=n_1,m_{10}=m_{10},\hat m_{11}=m_{11}.$
Further, consider the event
	$$\hat\cF(\vN,\vM)=\cbc{\hat n_\star=n_\star,\hat n_1=n_1,\hat m_{10}=\hat m_{01}=m_{10},\hat m_{11}=m_{11},
		\hat m_{00}=2m-2m_{10}-m_{11}}.$$
Additionally, set
\begin{equation}\label{eq:zeta}
	\zeta=\zeta(d,k)=(1-(k-1)q)^{3/2}
	\exp(-d/2-d^2/4).
\end{equation}
Finally, let $\Gamma_{n,m}(\vN,\vM)$ be the set of all graphs $G$ on vertex set $[n]$ with $m$ edges such that $\vN(G)=\vN$ and $\vM(G)=\vM$.

\begin{theorem}\label{thm:contig2}
Let $k\geq3, d>d_k$ and let $\xi>0$.
Then uniformly for all integer vectors  $\vN=(n_\star,n_1)$ and $\vM=(m_{10},m_{11})$
such that $m_{11}$ is even
and \eqref{eqthm:contig2} holds,
we have
  \begin{align}
	\pr\brk{\cF(\vN,\vM)|\hat\cF(\vN,\vM)}&\sim \zeta >0.
  \end{align}
Furthermore, given $\cF(\vN,\vM)$, $\hat\G$ is uniformly distributed on  $\Gamma_{n,m}(\vN,\vM)$.
\end{theorem}

Since $\hat n_\star,\hat n_1$ and $\hat m_{ab}$, $a,b\in\{0,1\}$ are sums of independent random variables,
it is easy to work out that under the assumption (\ref{eqthm:contig2}) we have $\pr\brk{\hat\cF(\vN,\vM)}=\Theta(n^{-1})$.
Further, \Thm~\ref{thm:contig2} shows that given the event $\hat\cF(\vN,\vM)$ the algorithm {\tt Forge} succeeds
with a probability $\zeta+o(1)$ that is bounded away from $0$ and, crucially,
given success the resulting random graph is perfectly uniformly distributed over the set of all graphs with $k$-core parameters $\vN,\vM$.
In effect, \Thm~\ref{thm:contig2} makes it easy to study the  random graph $\G$ given the order and size of its $k$-core.

In addition, since $\hat\G$ is uniform on $\Gamma_{n,m}(\vN,\vM)$ given $\cF(\vN,\vM)$,
in order to calculate the size of the set $\Gamma_{n,m}(\vN,\vM)$ we just need to compute the entropy of the output distribution of {\tt Forge} given $\cF(\vN,\vM)$.
This is fairly straightforward
because the construction involves a great degree of independence.
As we shall see in the next section this argument directly yields \Thm~\ref{Thm_LLT},  the multivariate local limit theorem.

\section{Proof strategy}\label{Sec_strategy}

\noindent
The main task is to prove \Thm~\ref{thm:contig2}, whence \Thm s~\ref{thm:contig} and~\ref{Thm_LLT} follow fairly easily.
Although some diligence is required, the proofs are completely elementary and none of the arguments are particularly difficult.
Let us begin by verifying that $\hat\G$ is uniform on $\Gamma_{n,m}(\vN,\vM)$ given success, i.e.\ that the second statement of Theorem~\ref{thm:contig2} holds.

\begin{proposition}\label{Prop_uniform}
Given $\cF(\vN,\vM)$, $\hat\G$ is uniformly distributed on  $\Gamma_{n,m}(\vN,\vM)$.
\end{proposition}
\begin{proof}
Fix $\vN,\vM$, let $n_0=n-n_\star-n_1$, $m_{01}=m_{10}$ and $m_{00}=2m-2m_{10}-m_{11}$, set
	$$\vn=(n_0,n_\star,n_1),\qquad\vm=(m_{00},m_{01},m_{10},m_{11}) $$
and let $\hat \vn
=(\hat n_0,\hat n_\star,\hat n_1)$ and $\hat \vm=(\hat m_{00},\hat m_{01},\hat m_{10},\hat m_{11})$ be as in {\tt Forge}.
Further, fix
 $G\in\Gamma_{n,m}(\vN,\vM)$ and let $\vd=(d_{G,ab}(v))_{v,a,b}$ be the corresponding degree sequence of $G$ broken down to edge types.
Moreover, let
	\begin{align}
P_0&=\prod_{v\in\cN_0(G)}\pr\brk{\Po(\lambda_{00})=d_{G,00}(v)}\prod_{v\in\cN_0(G)}\pr\brk{\Po_{\leq k-2}(\lambda_{10})=d_{G,10}(v)},&	
	P_\star&=\prod_{v\in\cN_\star(G)}\pr\brk{\Po(\lambda_{01})=d_{G,01}(v)},\nonumber\\
	P_1&=\prod_{v\in\cN_1(G)}\pr\brk{\Po_{\geq k}(\lambda_{11})=d_{G,11}(v)}\pr\brk{\Po(\lambda_{01})=d_{G,01}(v)},&
	\Pi&=\prod_{v\in V(G),a,b\in\{0,1\}}d_{G,ab}(v)!,\nonumber\\
	P&=P_1P_\star P_0\Pi.\label{eqUni0}
	\end{align}
Let $\vec{\hat d}=(\hat d_{ab}(v))_{v,a,b}$ be the random vector created by step (2) of {\tt Forge} and let $\cF(\vd)=\cF(\vN,\vM)\cap\{\vec{\hat d}=\vd\}$.
Since $\{\vec{\hat d}=\vd\}\subset\hat\cF(\vN,\vM)$ by Fact~\ref{Fact_WP}, Bayes' rule gives
	\begin{align}
	\pr\brk{\vec{\hat d}=\vd\big|\cF(\vN,\vM)}&=\frac{\pr[\cF(\vd)|\hat\cF(\vN,\vM)]}{\pr\brk{\cF(\vN,\vM)|\hat\cF(\vN,\vM)}}
		=\frac{\pr[\cF(\vN,\vM)|\vec{\hat d}=\vd]
		\pr[\vec{\hat d}=\vd|\hat\cF(\vN,\vM)]}{\pr\brk{\cF(\vN,\vM)|\hat\cF(\vN,\vM)}}.\label{eqUni1a}
	\end{align}		
Further, once more because the vertex types can be read off the degree sequence $\vec d$ by Fact~\ref{Fact_WP},
	\begin{align}\label{eqUni1b}
	\pr\brk{\vec{\hat d}=\vd|\hat\cF(\vN,\vM)}&=
		\frac{\nu_0^{n_0}\nu_1^{n_1}\nu_\star^{n_\star}P_0P_1P_\star}{\pr\brk{\hat\cF(\vN,\vM)}}
			=\frac{\nu_0^{n_0}\nu_1^{n_1}\nu_\star^{n_\star}P_0P_1P_\star}{
				\bink n{\vec n}\nu_0^{n_0}\nu_1^{n_1}\nu_\star^{n_\star}\pr\brk{\hvm=\vm|\hvn=\vn}}
			=\frac P{\bink n{\vec n}\pr\brk{\hvm=\vm|\hvn=\vn}\Pi}.
	\end{align}
Combining (\ref{eqUni1a}) and (\ref{eqUni1b}), we obtain
	\begin{align}\label{eqUni1}
	\pr\brk{\vec{\hat d}=\vd\big|\cF(\vN,\vM)}&=
		\frac{P\cdot\pr[\cF(\vN,\vM)|\vec{\hat d}=\vd]}{\bink{n}{\vn}\Pi\pr\brk{\hvm=\vm|\hvn=\vn}
		\pr\brk{\cF(\vN,\vM)|\hat\cF(\vN,\vM)}}\ .
	\end{align}
Moreover, by double counting
	\begin{align}\label{eqUni2}
	\pr\brk{\hat\G=G|\vec{\hat d}=\vd,\cF(\vN,\vM)}&=\frac{\pr[\hat\G=G|\vec{\hat d}=\vd]}{\pr[\cF(\vN,\vM)|\vec{\hat d}=\vd]}=
		\frac{\Pi}{\pr[\cF(\vN,\vM)|\vec{\hat d}=\vd](m_{00}-1)!!(m_{11}-1)!!m_{10}!}.
	\end{align}
Combining (\ref{eqUni1}) and (\ref{eqUni2}), we find
	\begin{align}\label{eqUni3}
	\pr\brk{\hat\G=G|\cF(\vN,\vM)}&=\frac{P}{\bink n{\vn}(m_{00}-1)!!(m_{11}-1)!!m_{10}!\pr[\cF(\vN,\vM)|\hat{\vn}=\vn]}.
	\end{align}
Crucially, in the expression (\ref{eqUni0}) that defines $P$ the factorials cancel, whence $P$ depends on $\vN,\vM$ but not on $\vd$.
Therefore, so does the right hand side of (\ref{eqUni3}), which means that the expression is independent of $G$.
\end{proof}

As a next step, in \Sec~\ref{Sec_Olly} we calculate the success probability of {\tt Forge}, confirming the first statement of Theorem~\ref{thm:contig2}, which is thus immediate from  Propositions~\ref{Prop_uniform} and~\ref{Prop_success}.

\begin{proposition}\label{Prop_success}
Suppose that $k\geq 3$, $d>d_k$ and let $\xi>0$.
Assume that $\vN,\vM$ are such that (\ref{eqthm:contig2}) holds and that $m_{11}$ is even.
Then uniformly $\pr\brk{\cF(\vN,\vM)|\hat\cF(\vN,\vM)}\sim \zeta.$
\end{proposition}

\noindent
The proof of \Prop~\ref{Prop_success} is based on the insight that given $\hat\cF(\vN,\vM)$ the algorithm is very likely to succeed
unless the random graph $\hat\G$ contains certain small substructures.
For example, in order to calculate the probability that $\hat\G$ is simple we just need to calculate the probability that
the random matchings from step (4) produces
 loops or multiple edges, a standard computation.
Similarly, it emerges that the most likely reason for step (7) to fail is the existence of certain bounded-sized subgraphs
within the subgraph of $\hat\G$ induced on $\hat\cN_0\cup\hat\cN_\star$, an event whose  probability we calculate by the method of moments.
The only aspect that requires a bit of technical work is ruling out troublesome sub-structures of intermediate sizes (unbounded but of lower order than $n$).

Further, in \Sec~\ref{section_LLTproofs} we use \Prop s~\ref{Prop_uniform} and~\ref{Prop_success} to determine 
$|\Gamma_{n,m}(\vN,\vM)|$  asymptotically.

\begin{proposition}\label{Prop_entropy}
Suppose that $k\geq 3$, $d>d_k$. Let $\xi>0$ and let $Q$ be the matrix from (\ref{eqQmatrix}).
{Then $Q$ is regular.}
Moreover, let $\vN,\vM$ be such that (\ref{eqthm:contig2}) holds and that $m_{11}$ is even.
Then uniformly
\begin{align*}
	|\Gamma_{n,m}(\vN,\vM)|&	\sim 
    \frac{1}
    {2\pi^2 d^2n^{2}\sqrt{\det Q}}
    \exp\br{-\frac n2
    \bck{Q^{-1}\Delta(\vN,\vM), \Delta(\vN,\vM)}}
	\bink{\bink n2}m.
\end{align*}
\end{proposition}

\noindent
The proof of \Prop~\ref{Prop_entropy} requires not much more than writing out the number of possible outcomes of $\hat\G$ given the event $\hat\cF(\vN,\vM)$
and applying Stirling's formula to obtain an asymptotic formula.
\Thm~\ref{Thm_LLT} is immediate from 
\Prop~\ref{Prop_entropy}.

\section{Proof of \Prop~\ref{Prop_success}}\label{Sec_Olly}

\noindent{\em Throughout this section we keep the assumptions of \Prop~\ref{Prop_success}.}

\subsection{Overview}
We prove Proposition~\ref{Prop_success} by calculating the success probability of steps~(5) and~(7) of {\tt Forge}.
To determine the success probability of step~(7), we need to calculate the probability
that running Warning Propagation on $\hat\G$ results in messages $\mu(\hat\G)$ that match the ``pseudo-messages'' $\vec{\hat{\mu}}$.
In Section~\ref{Sec_Flippingstructures} we will identify certain minimal structures,
called \emph{flipping structures}, which may cause this to fail.
Indeed, we show that \whp\ any flipping structure present is of a particular form, called a \emph{forbidden cycle}.
Hence, the success probability is asymptotically the same as the probability that no forbidden cycles are present.
Finally in Section~\ref{sec:successprob} we calculate the probability that $\hat{\vec G}$ is simple and 
contains no forbidden cycle.

The construction of $\hat\G$ is nothing but an enhanced configuration model.
Specifically, each vertex $v\in[n]$ receives $\hat d_{ab}(v)$ \emph{half-edges of type $ab$} for $a,b\in\{0,1\}$ and
step (4) of {\tt Forge} is a uniform matching of these  half-edges that respects the types.
To be precise, half-edges of type $00$ get matched to other half-edges of type $00$, and analogously for half-edges of type $11$.
Moreover, half-edges of type $01$ are matched to half-edges of type $10$ and vice versa. 
Each pair of matched half-edges induces an edge of the random multi-graph $\hat\G$.
We orient the edges of $\hat\G$ that result from the matching of $01$ and $10$ half-edges from $01$ to $10$.
Thus, $\hat\G$ contains some undirected edges (resulting from $00$ and $11$ half-edges) and some directed ones.
Further, let
	$$\setB=\cbc{v\in \hat\cN_0:\hat d_{10}(v)=k-2},\qquad\sizeB=|\setB|.$$
In addition, we define the events
	\begin{align*}
	\cE_1&=\cbc{\mbox{$\hat\G$ is simple (i.e.\ contains no loops or multiple edges)}},&
	\cE_2&=\cbc{\mbox{$\hat\G[\hat \cN_\star]$ contains no directed cycle}},\\
	\cE_3&=\cbc{\mbox{$\hat\G[\setB]$ contains no cycles}},&
	\cE&=\cE_1\cap\cE_2\cap\cE_3.
	\end{align*}
Moreover, we recall  from Section~\ref{Sec_Forge} that for given integer vectors $\vN=(n_\star,n_1)$ and 
$\vM=(m_{10},m_{11})$ such that $m_{11}$ is even, $\cF(\vN,\vM)$ denotes 
the event that {\tt Forge} succeeds and $\hat n_\star=n_\star, \hat n_1=n_1, \hat m_{10} =m_{10},
\hat m_{11}=m_{11}$, while
	$$\hat\cF(\vN,\vM)=\cbc{\hat n_\star=n_\star,\hat n_1=n_\star,\hat m_{10}=\hat m_{01}=m_{10},\hat m_{11}=m_{11},
		\hat m_{00}=2m-2m_{10}-m_{11}}.$$
We break the proof of Proposition~\ref{Prop_success} down into the two steps summarised by the following two propositions.

\begin{proposition}\label{lem:condsprob}
Let $\delta>0$ be any constant. Uniformly for all ${\vN},{\vM}$ such that $m_{11}$
is even and \eqref{eqthm:contig2} holds, we have
\begin{align*}
\pr \brk{\cE_2\;|\;\hat\cF(\vN,\vM)} &\sim 1-(k-1)q,&
\pr \brk{\cE_3 \; |\; \hat\cF(\vN,\vM)} &\sim\sqrt{1-(k-1)q},\\
\pr \brk{\cE_1\; | \; \cE_2\cap\cE_3 \cap \hat\cF(\vN,\vM)} &\sim\exp\left(-\frac{d}{2}-\frac{d^2}{4}\right).
\end{align*}
Furthermore, conditioned on $\hat\cF(\vN,\vM)$, the events $\cE_2$ and $\cE_3$ are independent, so
$$
\pr\brk{\conds|\hat\cF(\vN,\vM)}\sim(1-(k-1)q)^{3/2}
\exp\left(-\frac{d}{2}-\frac{d^2}{4}\right).
$$
\end{proposition}

\begin{proposition}\label{Prop_cycl}
Uniformly for all ${\vN},{\vM}$ such that $m_{11}$ is even and \eqref{eqthm:contig2} holds, we have
$$\pr\brk{\cF(\vN,\vM) |\hat \cF (\vN,\vM)}\sim\pr\brk{\cE |\hat \cF(\vN,\vM)}.$$
\end{proposition}

\noindent
After formally introducing flipping structures in \Sec~\ref{Sec_Flippingstructures}
and investigating the subgraph $\hat\G[\hat\cN_0]$ in \Sec~\ref{Sec_subc},
we will prove \Prop~\ref{lem:condsprob} in Section~\ref{sec_cycl} and
Proposition~\ref{Prop_cycl} in Section~\ref{sec:wellconstructed}.
Proposition~\ref{Prop_success} follows immediately from
Propositions~\ref{lem:condsprob} and~\ref{Prop_cycl}.

\subsection{Flipping structures}\label{Sec_Flippingstructures}
Recall that $\hat\cN_0,\hat\cN_\star,\hat\cN_1$ denote the 
random partition of $[n]$ constructed in step (1) of 
{\tt Forge}.
Further recall that given success in step (5), in step (6)
for $(v,w)\in[n]\times[n]$ we defined pseudo-messages
		\begin{align*}
		\hat{\vec\mu}_{v\to w}
		&=\vecone\{v\in\hat\cN_1,w\in\partial_{\hat\G}v\}
		+\vecone\{v\in\hat\cN_\star,
		\exists i,j:\{(v,0,1,i),(w,1,0,j)\}\in\hat\cM_{10}\}
		\end{align*}
and our aim is to calculate the probability that $\mu(\hat\G)=\hat{\vec\mu}$.
We begin with some basic observations. 

\begin{fact}\label{fact_upperbound}
If $\hat\G$ is simple, then 
$\hat{\vec\mu}_{v\to w}\leq \mu_{v\to w}(\hat\G)$ for all $(v,w)\in[n]\times[n]$.
\end{fact}
\begin{proof}
A straightforward induction shows that $\hat{\vec\mu}_{v\to w}\leq \mu_{v\to w}(t|\hat\G)$ for all $t\geq0$.
\end{proof}

\noindent
In contrast to $\hat\cN_0,\hat\cN_\star,\hat\cN_1$, which are defined in terms of the pseudo-messages $\hat{\vec\mu}$, the partition
$\cN_0(\hat\G)$, $\cN_\star(\hat\G),$ $\cN_1(\hat\G)$ is induced by the 
actual Warning Propagation messages on $\hat\G.$

\begin{fact}\label{claim_disc}
If $\hat\G$ is simple, then we 
have $\hat{\vec \mu}=\mu(\hat\G)$ if and only if
$\hat \cN_x=\cN_x(\hat\G)$ for all $x\in\{0,\star,1\}$.
\end{fact}
\begin{proof}
The construction of $\hat\G$ guarantees that $\hat d_{xy}(v)$ equals the number of neighbours $w$ of $v$ in $\hat\G$
such that $\hat\vmu_{w\to v}=x$ and $\hat\vmu_{v\to w}=y.$
Hence,
	\begin{align*}
		\hat\cN_0&=\textstyle\cbc{v:\sum_{u\in\partial v}
		\hat\vmu_{u\to v}(G)\leq k-2},&
	\hat\cN_\star&=\textstyle\cbc{v:\sum_{u\in\partial v}
	\hat\vmu_{u\to v}(G)= k-1},&
	\hat\cN_1&=\textstyle\cbc{v:\sum_{u\in\partial v}
	\hat\vmu_{u\to v}(G)\geq k},
	\end{align*}
and thus the assertion is immediate from Fact~\ref{fact_upperbound}.	
\end{proof}

Suppose that $\hat\G$ is simple but $\hat{\vec \mu}\neq \mu(\hat\G)$.
By Fact~\ref{claim_disc} there is $x\in\{0,\star,1\}$ with  $\hat \cN_x\neq \cN_x(\hat\G).$
We would like to identify a minimal structure that is ``responsible'' for the discrepancy.
To this end we introduce a modified version of Warning Propagation.
Let us write $\edgesethat$ for the set of ordered pairs of  adjacent vertices in $\hat\G$
	(i.e., $\edgesethat$ contains the pairs $(v,w)$, $(w,v)$ iff $v,w$ are connected by an edge in $\hat\G$).
For a subset $S\subset\edgesethat$ we define the modified Warning Propagation with messages
 $\mu_{v\to w}(t|\hat \G,S)$ and marks  $\mu_{v}(t|\hat \G,S)$ 
as follows.
Initially, we set
$$
\mu_{v\to w}(0|\hat \G,S)=
\begin{cases}
1 & \text{if } \hat\vmu_{v\to w}=1 \text{ or } (v,w)\in S,\\
0 & \text{otherwise.}
\end{cases}
$$
In other words, we initialise according to the pseudo-messages, except possibly on $S$,
where all messages are initially $1$.
Further, we use 
 the same update rules~\eqref{WP_messageupdate}
as in Section~\ref{Sec_WP}, namely
\begin{equation*}
\textstyle \mu_{v\to w}(t+1|\hat\G,S)
 =\vecone\cbc{\sum_{u\in\partial_{\hat \G} v\setminus w}
 \mu_{u\to v}(t|\hat\G,S)\geq k-1}\qquad\mbox{for integers }t\geq0.
\end{equation*}
Additionally, the mark of $v\in [n]$ is defined as
\begin{equation*}
\textstyle \mu_v(t|\hat\G,S)=\vecone\cbc{\sum_{u\in\partial_{\hat\G} v}
\mu_{u\to v}(t|\hat\G,S)\geq k}\qquad\mbox{for integers }t\geq0.
\end{equation*}
As in the original Warning Propagation algorithm, all messages are monotonically decreasing
and we set 
	$$\mu_{v\to w}(\hat \G,S) = \lim_{t\to \infty} \mu_{v\to w}(t|\hat \G,S).$$
Furthermore, let
	\begin{align*}
	\hat\cN_0(S)&=\textstyle\cbc{v:\sum_{u\in\partial v}\mu_{u\to v}(\hat \G,S)\leq k-2},\\
	\hat\cN_\star(S)&=\textstyle\cbc{v:\sum_{u\in\partial v}\mu_{u\to v}(\hat \G,S)= k-1},\\
	\hat\cN_1(S)&=\textstyle\cbc{v:\sum_{u\in\partial v}\mu_{u\to v}(\hat \G,S)\geq k}.
	\end{align*}
We make three simple but important observations.

\begin{fact}\label{Fact_flip1}
\begin{enumerate}
\item $\hat{\cN}_x(\emptyset) = \hat \cN_x$ for all $x\in\{0,\star,1\}$.
\item $\hat{\cN}_x(\edgesethat)=\cN_x(\hat \G)$ for all $x\in\{0,\star, 1\}$.
\item $\hat \cN_1 \subset \hat \cN_1(S)\subset \cN_1(\hat \G)$ and $\hat \cN_1 \cup \hat \cN_\star \subset \hat \cN_1(S)\cup \hat \cN_\star(S)\subset \cN_1(\hat \G)\cup \cN_\star(\hat\G)$ for any $S\subset\edgesethat$.
\end{enumerate}
\end{fact}
\begin{proof}
To obtain the first claim we observe that $\mu_{v\to w}(0|\hat\G,\emptyset)=\hat{\vec\mu}_{v\to w}$ and that by construction 
$\hat{\vec\mu}$ is a fixed point of the 
modified Warning Propagation algorithm for $S=\emptyset$, i.e. 
 $\mu_{v\to w}(\hat\G,\emptyset)=\hat{\vec\mu}_{v\to w}$
 for all $v,w$.
With respect to the second assertion, since $\mu_{v\to w}(0|\hat\G,\edgesethat)=1$ for all $v,w$, we have 
$\mu_{v\to w}(\hat\G,\edgesethat)=\mu_{v\to w}(\hat\G)$ for all $v,w$.
{The third assertion is immediate from Fact~\ref{fact_upperbound}.}
\end{proof}

\begin{definition}
A \emph{flipping structure} of $\hat\G$ is an inclusion-minimal set $S\subset\edgesethat$ such that there exists $x\in\{0,\star,1\}$ such that $\hat \cN_x \neq \hat \cN_x(S)$.
\end{definition}

\noindent
Facts~\ref{Fact_flip1} shows that, unless $\hat \cN_x \neq \cN_x(\hat \G)$ for all $x\in\{0,\star,1\}$, there exists a flipping structure.

Hence, we are left to calculate the probability that $\hat\G$ contains a flipping structure.
To this end we point out a few (deterministic) properties of a flipping structure.
Let $\auxH$  be the  set  of all pairs  $(v,w)\in\edgesethat$ with $\hat\vmu_{v\to w}=1$.
{Recall that we oriented the edges within $\hat\G[\hat \cN_\star]$.}
For a set $S\subset\eKathrin$ let $V(S)$ be the set of vertices $v\in[n]$ such that there is a neighbour $w$ of $v$ in $\hat\G$ with
$(v,w)\in S$ or $(w,v)\in S$. We denote by $\Sdir$  the directed graph
on vertex set $V(S)$ and edge set $S$ and 
let $\delta^-(\Sdir),\delta^+(\Sdir)$  be
the minimum in- and out-degree of this directed graph.
Similarly, denote by $\Sundir$ the undirected graph  on $V(S)$ with edge set $\{\{v,w\}: (v,w)\in S\}$.

\begin{proposition}\label{Prop_fs}
{Given that $\hat\G$ is simple,}
any flipping structure $S$ of $\hat\G$ enjoys the following eight properties. 
\begin{enumerate}[(i)]
\item $\auxH\cap S = \emptyset$.\label{claim_fs-3}
\item For any edge $\{u,v\}$ we have $\mu_{v\to w}(\hat \G,S)=\vecone\{(v,w)\in \auxH\cup S\}$.
In other words, the initialisation of the modified Warning Propagation algorithm
with input $S$ is already a fixed point.
\label{claim_fs-2}
\item 
$\Sdir$ is strongly connected -- in particular,
$\delta^-(\Sdir),\delta^+(\Sdir)\ge 1$.\label{claim_fs-1}
\item 
Either $S\subset \hat\cN_0\times\hat\cN_0$ or $S\subset\hat\cN_\star\times\hat\cN_\star.$
\label{claim_fs2}
\item If $S\subset\hat\cN_\star
\times \hat\cN_\star$, then $\Sdir$ forms a directed cycle
in $\hat\G[\hat\cN_\star]$.\label{claim_fs1}
\item If $S\subset\hat\cN_+\times\hat\cN_+$ then $\Sundir$ 
forms a cycle in $\hat\G[\hat\cN_+]$.
\label{claim_fs3}
\item 
Any vertices of $\Sundir$ in $\hat \cN_0\setminus \setB$ have at least $3$ distinct
neighbours in
$\Sundir$.\label{claim_fs4}
\item Any vertices of $\Sundir$have at least $2$ distinct
neighbours in $\Sundir$.\label{claim_fsnew}
\end{enumerate}
\end{proposition}

\begin{proof}
For $S\subset \edgesethat$ let  
	$$d_S^-(v)=|\{w : \mu_{w\to v} (\hat\G,S)=1\}|,\qquad d_S^+(v)=|\{w : \mu_{v\to w} (\hat\G,S)=1\}|.$$
\begin{enumerate}[(i)]
\item 
This simply follows from the minimality of $S$, since an edge of $\auxH$ would be
initialised with a message of $1$ in the modified Warning Propagation algorithm regardless of whether it lies in $S$ or not.
\item
Since the messages of the modified Warning Propagation algorithm 
are monotonically decreasing, we have 
$\mu_{v\to w}(\hat \G,S)\leq \vecone\{(v,w)\in \auxH\cup S\}$.
Further, 
by construction 
$\hat{\vec\mu}$ is a fixed point of the 
modified Warning Propagation algorithm for $S=\emptyset.$
Therefore, for $(v,w)\in\auxH$ we have 
$\mu_{v\to w}(\hat \G,S)\geq\mu_{v\to w}(\hat \G,\emptyset)=
\hat\vmu_{v\to w}=1.$ 
Let $S'$ consist of those directed edges $(v,w)\notin \auxH$ such that $\mu_{v\to w}(\hat \G,S)=1$. Then $S' \subset S$ and for any $v,w$, $\mu_{v\to w}(\hat \G,S')=\mu_{v\to w}(\hat \G,S)$. By the minimality of $S$ we have $S=S'$.
\item
Suppose there is a partition $X\dot\cup Y$
of the vertex set of $\Sdir$ such that $X$ and $Y$ are both non-empty
and there are no edges in $\Sdir$ from $X$ to $Y$. Then let 
$S'=\{(v,w)\in S: v,w\in Y\}$. For any $y \in Y$ and $v\in V(\hat \G)$ we have $\mu_{v\to y}(\hat \G,S)=\mu_{v\to y}(\hat \G,S')$, and therefore also
$\mu_{y\to v}(\hat \G,S)=\mu_{y\to v}(\hat \G,S')$. In other words, $X$ has no effect on the
messages sent out by $Y$. But then $S'$ would be a smaller flipping structure, contradicting the minimality of $S$.
\item
By (i) no edge $(v,w)$ where $v\in \hat \cN_1$ lies in $S$, for such a directed
edge lies in $\auxH$. But since $\delta^+(\Sdir)\ge 1$ by~\eqref{claim_fs-1}, no vertex
of $\hat \cN_1$ can lie in $S$.
Similarly, for any $u\in \hat \cN_\star$ and $v\in \hat \cN_0$ we have $(u,v)\in \auxH$ and therefore $(u,v)\notin S$.
Thus the result follows by~\eqref{claim_fs-1}.
\item
By construction a vertex $v\in \hat\cN_\star$ has 
$d^-_{\emptyset}(v)=k-1$. By~\eqref{claim_fs-1}, $\Sdir$ contains a directed cycle. On the other hand, if $S'\subset S$ is such that $S'$ forms
a directed cycle within $\hat \cN_\star$, then for each $v\in S'$ we have
$d^-_{S'}(v)\ge k$, meaning $v\in \hat \cN_1(S')$. Therefore by the minimality
of $S$ we have $S=S'$.
\item
By~\eqref{claim_fs-1}, $\Sdir$ must contain a directed cycle. On the other hand, if $S'\subset S$ 
forms a directed cycle, then for $v\in S'$ we have $d^-_{S'}(v)=k-1$. Therefore such vertices
are in $\hat \cN_\star(S')$ and by the minimality of $S$ we have $S=S'$
and the assertion follows since $S'$ forms a cycle in 
$\Sundir$.
\item
Let $v\in  \hat \cN_0 \setminus \setB$ be a vertex in $\Sundir$, then it holds that $d^-_{\emptyset}(v)\le k-3$. If $v$ has only one 
in-neighbour in $\Sdir$, then by~\eqref{claim_fs-2} we have
$d^-_{S}(v)\le k-2$ and $\mu_{v\to w}(\hat \G,S)=0$ for all neighbours $w$ of $v$ in $\hat\G$, i.e. $d_S^+(v)=0$
so by~\eqref{claim_fs-2}, we obtain that $v$
has no out-neighbour in $\Sdir$ and therefore $\delta^+(\Sdir)=0$. 
But this contradicts~\eqref{claim_fs-1}.
Therefore, $v$ has at least $2$ in-neighbours in $\Sdir$. By~\eqref{claim_fs-1}, $v$ has at least one out-neigbhour in $\Sdir$.
Now we just need to exclude the possibility that equality holds in both cases and one of the in-neighbours 
of $v$ in $\Sdir$ is also the out-neighbour. For if equality holds, 
i.e. $v$ has exactly two in-neighbours, then
we have $d^-_{S}(v)\leq k-1$. But this means that 
if $w$ is such that $\mu_{w\to v}(\hat \G,S)=1,$
then $\mu_{v\to w}(\hat \G,S)=0$. That is, no vertex 
$w$ can simultaneously be in- and out-neighbour of $v$, as required.
\item 
Let $v\in\hat\cN_+$ be a vertex in $\Sundir$, so $d^-_{\emptyset}(v)=k-2$. Assume that $v$
does only have one neighbour $w$ in $\Sundir$. By~\eqref{claim_fs-1}
$w$ is an in- and out-neighbour of $v$ in $\Sdir$. By~\eqref{claim_fs-2},
in this case we have that 
$d^-_{S}(v) = k-1$, so again 
 we can never have
$\mu_{v\to w}(\hat \G,S) = \mu_{w\to v}(\hat \G,S)=1$.
\qedhere
\end{enumerate}
\end{proof}

In light of \Prop~\ref{Prop_fs} (v) and (vi) we call a flipping structure $S$ a {{\em forbidden cycle}} if 
either $S\subset\hat\cN_\star\times \hat\cN_\star$ or $S\subset\hat\cN_+\times\hat\cN_+$.

\subsection{The subgraph $\hat\G[\hat\cN_0]$}\label{Sec_subc}
We proceed to analyse the structure of the induced subgraphs  $\hat\G[\hat\cN_0]$ and $\hat\G[\hat\cN_+]$ to facilitate the proofs of \Prop s~\ref{lem:condsprob} and~\ref{Prop_cycl}.
{We condition on the event $\cE_1\supset\cE$ that $\hat\G$ is simple.}
The following lemma determines the precise distribution of  $\hat\G[\hat\cN_0]$ given $\hat\cF(\vN,\vM)\cap\cE_1$.

\begin{lemma}\label{lem_Nzero}
Let $\vN, \vM$ be such that $m_{11}$ is even and \eqref{eqthm:contig2} holds. 
Given $\hat\cF(\vN,\vM)\cap\cE_1$  the induced subgraph  $\hat\G[\hat\cN_0]$ 
is a uniform random graph on $\hat n_0$ vertices with $\hat m_{00}/2$ edges.
\end{lemma} 
\begin{proof}
Given $\hat\cF(\vN,\vM)$, $\hat\G[\hat\cN_0]$ clearly has $\hat n_0$ vertices.
Further, by step (2) of {\tt Forge} we have $\hat d_{00}(v)=0$
for all $v\not\in\hat\cN_0.$ That is, all $\hat m_{00}$ half-edges of type $00$ are assigned to vertices in $\hat\cN_0.$
Given $\hat m_{00}$ each such half-edge is assigned to a vertex in $\hat\cN_0$ uniformly at random,
and subsequently $\hat\G[\hat\cN_0]$ is formed by matching the half-edges randomly.
In effect, given $\cE_1$ the random graph $\hat\G[\hat\cN_0]$ is uniformly distributed.
\end{proof}

\begin{corollary}\label{claim:avdeg}
For any $\delta >0$ there exists $\eps=\eps(\delta,d,k)>0$ such that for all $\vN, \vM$ such that $m_{11}$ is even and \eqref{eqthm:contig2} holds the following is true.
	\begin{quote}Given $\hat\cF(\vN,\vM)\cap\cE_1$, \whp~ $\hat\G[\hat\cN_0]$ does not contain a subgraph on fewer than $	\eps n$ vertices with average degree at least 
	 $2(1+\delta)$.
	 \end{quote}
\end{corollary}
\begin{proof}
Since a sparse uniformly random graph is well-known to feature no small subgraphs of average degree strictly greater than two,
the assertion is immediate from \Lem~\ref{lem_Nzero}.
\end{proof}

\begin{corollary}\label{claim:avdeg2}
For any $d,k$ there exists $\delta(d,k)>0$ such that for all $\vN, \vM$ such that $m_{11}$ is even and \eqref{eqthm:contig2} holds, the following is true.
	\begin{quote}Given $\hat\cF(\vN,\vM)\cap\cE_1$, \whp~ $\hat\G[\hat\cN_0]$ does not contain a pair of disjoint non-empty subsets $S,T\subset\hat\cN_0$ such that $|S|\leq\delta|T|$ and such that every vertex in $T$ has at least two neighbours in $S$.
	 \end{quote}
\end{corollary}
\begin{proof}
We claim that the probability that there exist such sets $S,T$ of sizes $s,t$ is bounded by
	\begin{align*}
	\bink ns\bink nt\bcfr{O(s)}n^{2t},
	\end{align*}
with the $O(\nix)$-term depending on $d$.
Indeed, the binomial coefficients bound the number of ways of choosing $S,T$.
Due to monotonicity we may bound the probability term via the binomial random graph of bounded average degree, and thus the probability that a given $v\in T$ has two neighbours in $S$ is bounded by $(O(s)/n)^2$.
Further,
	\begin{align*}
	\bink ns\bink nt\bcfr{O(s)}n^{2t}\leq\bcfr{\eul n}{s}^s\bcfr{\eul n}{t}^t\bcfr{O(s)}n^{2t}
	& \leq\exp(s+O(t))\bcfr{s}{t}^t\bcfr{s}{n}^{t-s} 
	 \le \left(O(\delta)\right)^t \left(\frac{t}{n}\right)^{t/2}.
	\end{align*}
Summing over all $s,t$, we obtain
\begin{align*}
\sum_t\sum_{s\le \delta t}\left(O(\delta^2)\frac{t}{n}\right)^{t/2}
& \le \sum_{t\le \ln n} \delta \ln n \frac{1}{\sqrt{n}} + \sum_{t\ge \ln n} n\left(O(\delta^2)\right)^{\ln n} = o(1),
\end{align*}
as desired.
\end{proof}

As a next step we establish that the subgraph induced on $\hat\cN_+$ is subcritical, i.e. has average degree less than $1$.
In effect,  there is no large component \whp\

\begin{lemma}\label{lem:subcritical}
Let ${\vN},{\vM}$  be such that $m_{11}$ is even and \eqref{eqthm:contig2} holds. 
Given $\hat\cF(\vN,\vM)\cap \cE_1$ the average degree of $\hat\G[\setB]$ converges in probability to
$
\gamma_+=\bar{q}(1-p)d =(k-1)q <1.
$
\end{lemma}

We proceed to prove \Lem~\ref{lem:subcritical}.
We recall that $\hat n_+=|\hat\cN_+|$ and further let $\hat m_+$ be the number  of edges spanned by $\hat\cN_+.$ 
Let $\hat \cF(\vN,\vM,n_+)=\hat\cF(\vN,\vM)\cap\{\sizeB=n_+\}$ and 
$\hat \cF(\vN,\vM,n_+,m_+)=\hat\cF(\vN,\vM)\cap\{\sizeB=n_+,\hat m_+=m_+\}$.
The following two claims facilitate the proof of \Lem~\ref{lem:subcritical}.

\begin{claim}\label{claim_nplus1}
Let ${\vN},{\vM}$  be such that $m_{11}$ is even and \eqref{eqthm:contig2} holds. 
Then $\hat n_+$ has distribution $\Bin(\hat n_0, \bar q)$.
Moreover, given $\hat \cF(\vN,\vM,n_+)$, $\hat m_+$ has distribution
$\Bin(\hat m_{00}/2, (\hat n_+/\hat n_0)^2)$.
Further, given $\hat\cF(\vN,\vM,n_+,m_+)\cap \cE_1$, $\hat\G[\hat\cN_+]$
is a uniformly random graph on $\hat n_+$ vertices with $\hat m_+$ edges.
\end{claim}
\begin{proof}
We recall that $\hat \cN_+$ is the set of all $v\in\hat\cN_0$
such that $\hat d_{10}(v)=k-2.$
By the definition of $\hat d_{10}$,
	$$\pr\brk{v\in \setB | v \in \hat\cN_0} = \pr\brk{ \hat d_{10}(v)= k-2 | \hat d_{10}(v)\le k-2} = \bar q$$
independently for all for all $v\in[n]$. 
Hence, given $\hat n_0$, the parameter $\sizeB $ has distribution  $\Bin (\hat n_0,\bar q)$.

Since $\hat\cN_+\subset\hat\cN_0$, all edges spanned by $\setB$ are of type $00$. 
Moreover, the construction in steps (2)--(3) of {\tt Forge} ensures that given $\hat n_+$ and  $\hat n_0$, for each of the $\hat m_{00}$ half-edges of type $00$ the probability of being assigned to a vertex in $\setB$ is just $\sizeB/{\hat n_0}$.
Further, each of the $\hat m_{00}/2$ edges constructed  from the matching of half-edges of type $00$ forms a 
edge within $\hat\G[\hat\cN_+]$ iff both of the corresponding half-edges were assigned to a vertex from $\hat\cN_+$.
Therefore, the number $\hat m_+$ of edges  within $\setB$ is distributed as $\Bin(\hat m_{00}/2,(\sizeB/\hat n_0)^2)$. 
Finally, given $\hat m_+$, steps (5) and (6) of {\tt Forge} generate a random multi-graph on $\hat\cN_+$ and given
the event $\cE_1$, this graph is uniformly distributed given its order and size by the same token as in the proof of \Lem~\ref{lem_Nzero}.
\end{proof}

\begin{claim}\label{claim_nplus}
Suppose that $\omega=\omega(n)\to\infty.$
Uniformly for all ${\vN},{\vM}$ such that $m_{11}$ is even and \eqref{eqthm:contig2} holds, 
given $\hat\cF(\vN,\vM)$ \whp~we have $|\sizeB-(1-p)\bar q n | \le \omega\sqrt{n}$.
\end{claim}
\begin{proof}
To estimate $\sizeB$
denote by $\hat\cA(\omega)$ the event that $|\sizeB-(1-p)\bar q n | \le \omega\sqrt{n}$ and
let $\hat\cF(\vN)=\{\hat n_\star=n_\star, \hat n_1=n_1\}$. 
By Claim~\ref{claim_nplus1} given $\hat n_0$, the parameter $\sizeB $ has distribution  $\Bin (\hat n_0,\bar q)$.
Hence,
	\begin{equation}\label{eqclaim_nplus1}
	\pr\brk{\hat\cA(\omega/2)|\hat\cF(\vN)}=1+o(1).
	\end{equation}

To prove the desired bound given $\hat\cF(\vN,\vM)$, consider the event 
	$$\hat\cD(\xi)=\cbc{|\hat n_\star - n\nu_\star|
		+|\hat n_1-n\nu_1 |+|\hat m_{10}-2m\mu_{10}|+|\hat m_{11}-2m\mu_{11}|\leq \xi \sqrt{n}}.$$
To estimate its probability, we calculate
\begin{align*}
 \Erw\brk{\chi_{10}(v)}
 =\frac{1}{1-p}\sum_{i\leq k-2} \frac{i(dp)^i}{i!\exp(dp)}
 =\frac{dp}{1-p}\pr\brk{\Po(dp) \leq k-3}
 =dp(1-\bar q),\\
 \Erw\brk{\chi_{11}(v)}
 =\frac{1}{p(1-q)}\sum_{i\geq k}\frac{ i(dp)^i}{i!\exp(dp)}
 =\frac{dp}{p(1-q)}\pr\brk{\Po(dp) \geq k-1}
 =\frac{dp}{1-q}.
\end{align*}
Recalling the definitions of $\mu_{10},\mu_{11},\nu_1, \nu_0$
{we obtain that $\Erw[\hat m_{10}|\hat\cF(\vN)]=2m\mu_{10}$
and $\Erw[\hat m_{10}|\hat\cF(\vN)]=2m\mu_{10}.$}
Given $\hat\cF(\vN)$, the parameters $\hat m_{10}$ and $\hat m_{11}$ are sums of independent random variables with a bounded second moment by the construction in step (2) of {\tt Forge}.
Thus, the central limit theorem shows that $\pr[\hat\cD(\xi)|\hat\cF(\vN)]=\Omega(1)$ for any fixed $\xi>0$.
Therefore, (\ref{eqclaim_nplus1}) implies that
\begin{align}\label{eq:sizeB}
\pr\brk{\hat\cA (\omega/2)| 
\hat\cF(\vN)\cap\hat\cD(\xi) }=1+o(1).
\end{align}
Furthermore, conditioned on $\hat\cF(\vN)$,
perturbing $\vM$ by at most $O(\sqrt{n})$ in each coordinate 
will change $\sizeB$ by at most $O(\sqrt{n})$.
This implies that
for $\vN,\vM$ such that \eqref{eqthm:contig2} holds we have 
\begin{align*}
\pr\brk{\hat\cA(\omega) | 
\hat\cF(\vN,\vM)}
=1+o(1)
\end{align*}
by \eqref{eq:sizeB}.
\end{proof}

\begin{proof}[Proof of \Lem~\ref{lem:subcritical}]
Let $\omega=\omega(n)\to\infty$ sufficiently slowly. Let 
$\hat\cA(\omega)$ be 
the event that 
$|\sizeB-(1-p)\bar q n | \le \omega\sqrt{n}$.
By Claim~\ref{claim_nplus1}, the number {$\hat m_+$} of edges within 
$\hat\G[\setB]$ is distributed as $\Bin(\hat m_{00}/2,(\sizeB/\hat n_0)^2)$.
Hence,
$${\Erw[\hat m_+
|\hat\cF(\vN,\vM)\cap\cE_1\cap\hat\cA(\omega)]=}\frac{\hat m_{00}}{2}\left(\frac{\sizeB}{\hat n_0}\right)^2 \sim\frac{(1-p)\bar q d}{2}\sizeB.$$
Claim~\ref{claim_nplus} shows that  given $\hat\cF(\vN,\vM)$, the event $\hat \cA(\omega)$ occurs  \whp\
The Chernoff bound therefore shows that conditioned on
$\hat\cF(\vN,\vM)\cap\cE_1$ we have 
$\hat m_+\sim\frac{(1-p)\bar q d}{2}\sizeB$ \whp\
Therefore \whp\ the average degree of $\hat\G[\cN_+]$ is 
$(1-p)\bar q d+o(1)$. The assertion thus follows from  Fact~\ref{fact_FPana} (2). 
\end{proof}
\begin{corollary}\label{Claim_claim_fssmall_path}
Let $\vN,\vM$ be such that $m_{11}$ is even and 
\eqref{eqthm:contig2} holds.
Then there exists $\eps=\eps(d,k)$ such that given 
$\hat\cF(\vN,\vM)\cap \cE_1,$ \whp~there is no set $T\subset\hat\cN_0$ with the following properties:
	\begin{enumerate}[(1)]
	\item $t=|T|\leq \eps n,$
	\item there are $0.99 |T|\leq y\leq 1.01 |T|$ edges in 
	$\hat\G[T]$,
	\item there are $s\geq 0.1 |T|$ vertex-disjoint paths of length at  
	least $2$ whose internal vertices lie 
	in $\hat\G[\hat\cN_+]\setminus T$
	and that each join two vertices in $T$.
	\end{enumerate}
\end{corollary}
\begin{proof}
Let us  define 
$\nu_+=(1-p)\bar q$ and $\nu_-=(1-p)(1-\bar q)$
and pick a slowly growing $\omega=\omega(n)\to\infty$.
By Claim~\ref{claim_nplus} and
\Prop~\ref{lem:condsprob},
because $\hat n_++\hat n_-=\hat n_0$ we have 
 	\begin{equation*}
	\pr\brk{|\sizeB-\nu_+ n |+ |\hat n_--\nu_- n |\le 3\omega\sqrt{n}|\hat \cF(\vN,\vM)\cap \cE_1}=1-o(1).
		\end{equation*}
Let $\cA(3\omega)$ denote the event that 
$|\sizeB-\nu_+ n |+ |\hat n_--\nu_- n |\le 3\omega\sqrt{n}$ holds.
By Claim~\ref{claim_nplus1}, the number {$\hat m_+$} of edges within 
$\hat\G[\setB]$ is distributed as $\Bin(\hat m_{00}/2,(\sizeB/\hat n_0)^2)$.
Hence,
$$\Erw[\hat m_+
|\hat\cF(\vN,\vM)\cap\cE_1\cap\hat\cA(3\omega)]=
\frac{\hat m_{00}}{2}\left(\frac{\sizeB}{\hat n_0}\right)^2 \sim\frac{\gamma_+}{2}\sizeB.$$
Further, the Chernoff bound implies 
that conditioned on 
$\hat\cF(\vN,\vM)\cap\cE_1\cap\hat\cA(3\omega)$
\whp~we have 
 	\begin{equation}
	|2\hat m_+-\gamma_+\hat n_+|\le \omega\sqrt{n}.\label{eq:approx+}
		\end{equation}
		
Let $Y(k_1,\ldots,k_s)$ denote 
the number of subsets $T\subset \hat\cN_0$ with properties (1) -- (3) of size $t$ with paths of lengths $k_1,\ldots,k_s$.
We aim to use the first moment method 
for $Y(k_1,\ldots,k_s)$ conditioned on
	$\hat \cB=\hat A(3\omega)\cap\hat \cF(\vN,\vM)\cap \cE_1$.

Since the appearance of the given subgraph is a monotone graph property,
by \Lem~\ref{lem_Nzero} it suffices to estimate the probability 
of the existence of a subgraph with properties (1)--(3) 
in the binomial random graph on $\hat n_0$ vertices with average degree
$\gamma_0=\hat m_{00}/\hat n_0$; we will merely lose a constant factor.
Therefore, conditioned on $\hat\cF(\vN,\vM)\cap\cE_1$ 
 the expected number  of sets $T\subset\hat\cN_0$ of size $t$ that span $y$ edges is approximated up to a constant factor by
\begin{align}\label{eq:boundD1}
D=D(t,y,s)=\bink{\hat n_0}{t}\bink{\bink t2}{y}
\br{\frac{\gamma_0}{\hat n_0}}^y
\leq 
\br{\frac{\eul \hat n_0}{t}}^t\br{\frac{\eul t^2}{2y}}^y
\br{\frac{\gamma_0}{\hat n_0}}^y.
\end{align}

Similarly, by Claim~\ref{claim_nplus1} conditioned on 
$\hat\cF(\vN,\vM,n_+,m_+)\wedge\cE_1$ 
the expected number of paths of lengths $k_1,\ldots,k_s$ in $\hat \cN_+\setminus T$ whose endpoints are adjacent to a vertex in $T$ is upper bounded up to a constant 
by 
\begin{align*}
B(k_1,\ldots,k_s)
\le\bink{\hat n_+}{s}
\prod_{i=1}^s
\brk{
\bink{\hat n_+}{k_i}k_i!\br{\frac{2\hat m_+}{\hat n_+^2}}^{k_i}
\br{\frac{\gamma_0 t}{\hat n_0}}^{2}} \le
\bink{\hat n_0}{s}
\prod_{i=1}^s
\brk{\br{\frac{2\hat m_+}{\hat n_+}}^{k_i}
\br{\frac{\gamma_0 t}{\hat n_0}}^{2}}.
\end{align*}
Let
$$B_+(k_1,\ldots,k_s)=\prod_{i=1}^s\gamma_+^{k_i}.$$

For $\vN,\vM$ such that \eqref{eqthm:contig2} holds and 
$\omega\to\infty$ slowly enough, 
we have $(\nu_0n/\hat n_0)^t\leq \exp(O(\omega t/\sqrt n)).$ 
Therefore, from \eqref{eq:boundD1}
we obtain that conditioned on $\hat \cB$
\begin{align}\label{eq:boundD}
D
\leq
\br{\frac{\eul\nu_0n}{t}}^t\br{\frac{\eul t}{2y}}^y
\br{\frac{\gamma_0t}{\nu_0n}}^y\exp\br{O\br{\frac{\omega t}{\sqrt n}}}.
\end{align}
Similarly, by \eqref{eq:approx+}, conditioned on $\hat\cB$ 
we have  
$(2\hat m_+/\gamma_+\hat n_+)^{k_i}\leq 
\exp(O(\omega k_i/\sqrt n))$. Therefore,
conditioned on $\hat \cB$ we have
\begin{align}\label{eq:boundB}
B(k_1,\ldots,k_s)&
\leq
\br{\frac{\eul\nu_0n}{s }}^s 
\br{\frac{\gamma_0 t}{\nu_0 n}}^{2s}
B_+(k_1,\ldots,k_s)
\exp\br{O\br{\frac{\omega\sum_i k_i}{\sqrt n}}}.
\end{align}

Note also that conditioned on $\hat\cB$, we have $\frac{\gamma_0}{\nu_0}\sim d >1$.
To apply the first moment method
for $Y(k_1,\ldots,k_s)$ we consider
two cases.
\begin{description}
\item[Case 1: $s\geq 2t$]
Denote by  $Y'(k_1,\ldots,k_s)$ the 
number of subsets $T\subset\hat\cN_0$
with properties (1)--(3) and $s\geq 2t.$
Using $y\leq 1.01 t\leq 3s,$
from \eqref{eq:boundD} and~\eqref{eq:boundB} we obtain
\begin{align*}
\Erw[Y'(k_1,\ldots,k_s)|\hat\cB]
&\leq
\eul^{s+y+t}
\br{\frac{\gamma_0}{\nu_0}}^{2s+y}
\br{\frac{t}{y}}^y
\br{\frac{t}{s }}^s 
\br{\frac{t}{n}}^{s+y-t}
B_+(k_1,\ldots,k_s)
\exp\br{O\br{\frac {\omega\sum_i k_i+\omega t}{\sqrt n}}}
\\
&\leq
\eul^{5s}
\br{\frac{\gamma_0}{\nu_0}}^{5s}
\br{\frac{100}{99}}^{3s}
\left(\frac12\right)^s
\br{\frac{t}{n}}^{s+y-t}
B_+(k_1,\ldots,k_s)
\exp\br{O\br{\frac {\omega\sum_i k_i+\omega t}{\sqrt n}}}
\end{align*}
Further, since $\gamma_+<1$, for 
$\omega\to \infty$ slowly enough   
we have 
$$\sum_{k_1,\ldots,k_s}B_+(k_1,\ldots,k_s)
\exp\br{O\br{\frac {\omega \sum_i k_i}{\sqrt n}}}=O(1).$$ 
Therefore
using $s-t+y\geq s-t\geq 0.5 s$
 we obtain that for $\eps>0$ small enough,
\begin{align}\label{eq:boundY'}
\Erw[Y'(k_1,\ldots,k_s)|\hat\cB]=o(1).
\end{align}
\item[Case 2: $s<2t$]
Denote by  $Y''(k_1,\ldots,k_s)$ the 
number of subsets $T\subset\hat\cN_0$
with properties (1)--(3) and $s< 2t.$
Using $y\leq 1.01 t$,
from \eqref{eq:boundD} and~\eqref{eq:boundB} we obtain
\begin{align*}
\Erw[Y''(k_1,\ldots,k_s)|\hat\cB]
&\leq
\eul^{5t}
\br{\frac{\gamma_0}{\nu_0}}^{6t}
\br{\frac{100}{99}}^{2t}
10^{2t}
\br{\frac{t}{n}}^{s+y-t}
B_+(k_1,\ldots,k_s)
\exp\br{O\br{\frac {\omega\sum_i k_i+\omega t}{\sqrt n}}}
\end{align*}
Similarly as in Case 1, from $\gamma_+<1$
and $s+y-t\geq 0.09t$, we obtain that 
for $\omega\to\infty$ slowly enough 
and $\eps>0$ small enough,
\begin{align}\label{eq:boundY''}
\Erw[Y''(k_1,\ldots,k_s)|\hat\cB]=o(1).
\end{align}
\end{description}
Finally, from \eqref{eq:boundY'}
and \eqref{eq:boundY''}
we obtain
$$\Erw[Y(k_1,\ldots,k_s)|\hat\cB]
=\Erw[Y'(k_1,\ldots,k_s)|\hat\cB]
+\Erw[Y''(k_1,\ldots,k_s)|\hat\cB]=o(1)$$
as desired.
\end{proof}

\subsection{Proof of \Prop~\ref{lem:condsprob}}
\label{sec:successprob}\label{sec_cycl}
Our aim is to determine the probability that we have no forbidden cycles in
$\cN_\star$ or $\setB$, and no loops or multiple edges. We do this by proving
that the number of such structures is approximately Poisson distributed with
the appropriate mean. For this we use the method of moments, that is, \Thm~\ref{Cor_mom}.

To this end, let $X_{\star,\ell}$ be the number of directed cycles of length $\ell$ 
in $\hat\cN_\star$, $X_{+,\ell}$ the number of cycles
of length $\ell$ in $\hat\cN_+$ and
define $X_\star = \sum_{\ell=1}^\infty X_{\star,\ell}$ and $X_+ = \sum_{\ell=1}^\infty X_{+,\ell}$.
Furthermore, define $Y, Z$ to be the number of loops and multiple edges in 
$\hat\G$ respectively. Our aim is to determine the (conditional) probability of the event
that $X_{\star}=X_{+}=Y=Z=0$.
Let
 $\hat \cF(\vN,\vM,n_+)=\hat\cF(\vN,\vM)\cap\{\sizeB=n_+\}$.
For $\omega\to \infty$, by  {Claim~\ref{claim_nplus}} assumption \eqref{eqthm:contig2}
implies that $|\hat n_+-(1-p)\bar q n|\leq \omega\sqrt n$ \whp

\begin{lemma}
Let $\omega\to \infty $. Further, let
$n_+$ be such that $|\hat n_+-(1-p)\bar q n|\leq \omega\sqrt n.$
Then, uniformly for all $\vN,\vM$  such that $m_{11}$ even and \eqref{eqthm:contig2} 
holds, we have
\begin{align*}
\Erw \brk{X_\star|\hat\cF(\vN,\vM,n_+)} & = -(1+o(1)) \ln (1-(k-1)q);\\
\Erw \brk{X_+|\hat\cF(\vN,\vM,n_+)} & = -\frac{1}{2}(1+o(1)) \ln (1-(k-1)q);\\
\Erw \brk{Y \; | \: X_\star=X_+=0,\hat\cF(\vN,\vM,n_+)} & = (1+o(1))\frac{d}{2};\\
\Erw \brk{Z \; | \: X_\star=X_+=0,\hat\cF(\vN,\vM,n_+)} & = (1+o(1))\frac{d^2}{4}.
\end{align*}
\end{lemma}

\begin{proof}
We begin with $X_\star$, and will consider $\Erw\brk{X_{\star,\ell}}$ for bounded $\ell \ge 1$
-- this expectation tends to $0$ exponentially as $\ell\to \infty$, justifying
our choice of only considering $\ell$ bounded.
We first calculate, for bounded 
$\ell \ge 1$, the expected
number of collections of $\ell$ cyclically ordered  vertices and $2\ell$ ordered half-edges
which could conceivably form a directed cycle in $\hat \cN_\star$: we have $(\hat n_\star)_\ell/\ell$ choices for the
cyclically ordered vertices. By construction, each such vertex has $k-1$ half-edges
of type $10$.
The number of half-edges of type $01$ at each vertex is asymptotically distributed
as
$\Po(\hat m_{01}/(\hat n_1 +\hat n_\star))$
independently for each vertex.
Thus the
expected number of potential directed cycles of length $\ell$ is asymptotically
$$
\frac{1}{\ell}(\hat n_\star)_\ell \left((k-1)\frac{\hat m_{01}}{\hat n_1 +\hat n_\star}\right)^\ell.
$$
Now given such a choice of vertices and half-edges, the probability that they form a
directed cycle (with this ordering) is the probability that the relevant half-edges
are matched to each other, which is $1/(\hat m_{01})_\ell$. Thus,
by~\eqref{eq!!} the expected
number of directed cycles of length $\ell$ is
\begin{align*}
\Erw \brk{X_{\star,\ell}|\cF(\vN,\vM,n_+)} 
= (1+o(1))\frac{1}{\ell}(\hat n_\star)_\ell \left((k-1)\frac{\hat m_{01}}{\hat n_1 +\hat n_\star}\right)^\ell \frac{1}{(\hat m_{01})_\ell}
 = (1+o(1))\frac{1}{\ell} \left((k-1)q\right)^\ell.
\end{align*}
Note that $(k-1)q <1$ 
by Fact~\ref{fact_FPana} (2), and so (approximating the sum over all bounded $\ell$ by the sum to infinity) the expected total number of directed cycles in $\hat\cN_\star$ is 
$$
\Erw \brk{X_\star|\cF(\vN,\vM,n_+)} = \sum_{\ell=1}^\infty\Erw \brk{X_{\star,\ell}|\cF(\vN,\vM,n_+)}  = -(1+o(1))\ln \left(1-(k-1)q\right).
$$

The arguments for $X_+$ are similar, although the calculations are slightly different. Conditioned on 
$\hat\cF(\vN,\vM,n_+)$,  
each vertex of $\setB$ has asymptotically 
$\Po(\hat m_{00}/\hat n_0)$ half-edges of type $00$, and
therefore for $v\in\setB$ we have
 $$\Erw \brk{\hat d_{00}(v)(\hat d_{00}(v)-1)|\hat\cF(\vN,\vM,n_+), 
 v\in\hat\cN_+}= \frac{\hat m_{00}^2}{\hat n_0^2}.$$
Now the expected number of sequences of $\ell$ cyclically ordered (in either direction) vertices and $2\ell$ half-edges that could conceivably form a cycle is approximately
$$
\frac{1}{2\ell}(\sizeB)_\ell \Erw \brk{\hat d_{00}(v)(\hat d_{00}(v)-1)|\hat\cF(\vN,\vM,n_+),v\in\hat\cN_+}^\ell = (1+o(1))\frac{1}{2\ell}\sizeB^\ell \left(\frac{\hat m_{00}}{\hat n_0}\right)^{2\ell},
$$
while the probability that such a potential cycle is present (i.e.\ that the appropriate half-edges are matched together) is
$$
\frac{1}{(\hat m_{00}-1)(\hat m_{00}-3)\ldots (\hat m_{00}-2\ell+1)} = (1+o(1))\hat m_{00}^{-\ell}.
$$
Thus, conditioned on $\hat\cF(\vN,\vM,n_+)$ we obtain
$$
\Erw \brk{X_{+,\ell}|\hat\cF(\vN,\vM,n_+)}
=(1+o(1))\frac{1}{2\ell} \left( \frac{\sizeB \hat m_{00}}{\hat n_0^2} \right)^\ell
= (1+o(1)) \frac{((1-p)\bar q d)^\ell}{2\ell}.
$$
Since \eqref{eq:qbarexp} and 
Fact~\ref{fact_FPana} imply that $(1-p)\bar q d<1$, as in the previous case we have
$$
\Erw \brk{X_+|\hat\cF(\vN,\vM,n_+)} = \sum_{\ell=1}^\infty \Erw \brk{X_{+,\ell}|\hat\cF(\vN,\vM,n_+)} = -(1+o(1))\ln (1-(1-p)\bar q d)
$$
as claimed.

It remains to determine the expected number of loops and multiple
edges given $\{X_\star=X_+=0\}\cap\hat\cF(\vN,\vM,n_+).$
Conditioned on this event there are no loops or multiple 
edges in  $\hat\cN_\star$ or $\setB$.
We therefore consider the probability of having other loops or multiple edges.
Let $Y_{0},Y_1$ denote the number of loops in $\hat\cN_0 \setminus \setB$ and $\hat\cN_1$ respectively.
Conditioned on $\hat\cF(\vN,\vN,n_+)$,
for $v\in\hat\cN_-$ we have 
that $\hat d_{00}(v)$ is asymptotically distributed as
$\Po (\hat m_{00}/\hat n_0)$, and so the expected number of loops is
\begin{align}\label{eq:loops0}
\Erw\brk{Y_0 \; | \; X_\star=X_+ =0,\hat\cF(\vN,\vM,n_+)} & = (1+o(1))\Erw \brk{Y_0|\hat\cF(\vN,\vM,n_+)}\nonumber\\
& =
(1+o(1))\hat n_-
\Erw\brk{\left.\binom{\hat d_{00}(v)}{2}\right| \hat\cF(\vN,\vM,n_+), v\in\hat\cN_0\setminus\hat\cN_+} \frac{1}{\hat m_{00}-1}\nonumber\\
& = (1+o(1))(1-p)(1-\bar q)n \frac{d^2 (1-p)^2}{2}\frac{1}{(1-p)^2dn}
\nonumber\\
& = (1+o(1))(1-p)(1-\bar q)d/2.
\end{align}
To determine the expected number of loops
in $\hat\cN_1$ we aim to determine the
asymptotic distribution of 
$\hat d_{11}(v)$ for $v\in\hat\cN_1.$
We have 
\begin{align}\label{eq:bb11}
\hat n_1\Erw[\Po_{\geq k}(\lambda_{11})]
=\hat n_1 \sum_{x\geq k} x\frac{(dp)^x}{x!\exp(dp)p(1-q)} 
=\hat n_1\frac{dp}{p(1-q)}
=(1+o(1))p(1-q)n\frac{dp^2}{p(1-q)}
=(1+o(1)) \hat m_{11}.
\end{align}
Conditioned on $\hat\cF(\vN,\vM,n_+)$
step (2) of {\tt Forge} can be described 
by the following balls and bins experiment.
Each of the $\hat m_{11}$ half-edges
is distributed uniformly among 
$\hat n_1$ vertices subject to the 
constraint that each vertex receives at least $k$ half-edges. 
By \eqref{eq:bb11}, we have that 
$\Erw[\Po_{\geq k}(\lambda_{11})]\sim
\hat m_{11}/\hat n_1$. Since this 
is the distribution with highest 
entropy and this expectation,
for $v\in\hat\cN_1$ we have that 
$\hat d_{11}(v)$
asymptotically distributed as  $\Po_{\ge k}(\lambda_{11})$. Therefore, we have 
\begin{align}\label{eq:loops1}
\Erw\brk{Y_1 \; | \; X_\star=X_+ =0,\hat\cF(\vN,\vM,n_+)} = (1+o(1))\Erw \brk{Y_1}
& =(1+o(1))\hat n_1 \Erw \brk{\left.\binom{\hat d_{11}(v)}{2}\right| v\in\hat\cN_1}\frac{1}{\hat m_{11}-1}\nonumber\\
& = (1+o(1))\frac{p(1-q)n}{2p^2dn} \sum_{x\ge k} x(x-1) \frac{(dp)^x}{x!\exp(dp)p(1-q)}\nonumber\\
& = (1+o(1))\frac{1}{2p^2d} (dp)^2 \pr\brk{\Po(dp)\ge k-2}\nonumber\\
&= (1+o(1))\frac{d}{2} (p+(1-p)\bar q).
\end{align}

Summing up the two contributions from~\eqref{eq:loops0} and~
\eqref{eq:loops1} we obtain
$$
\Erw\brk{Y \; | \: X_\star=X_+=0}=(1+o(1))\frac{d}{2}
$$
as claimed.

We now calculate the expected number of multiple edges. 
Assume that there is a multiple edge 
joining two vertices in $\hat\G.$
Then the types of the edges are determined by the end-vertices.
By construction, it either holds that 
both edges must result from the same matching in step (4) 
of {\tt Forge.}
Along these lines, we will say that a multiple edge is of type $11$, $00$ or $01/10$ respectively for each possible case.
Conditioned on $X_+=0$ there are no multiple edges of type $00$
such that both end-vertices lie in $\hat\cN_+.$
Further, conditioned on $X_\star=0$ there is no multiple edge
of type $01/10$ such that both edges are oriented in the same 
direction.
Denote 
by $Z_{00}$
the number of multiple edges of type $00$ which lie within $\hat\cN_-$,
by $Z_{11}$ the number of multiple edges
of type $11$ and by $Z_{01}$ the number of multiple edges of type $01/10$ in which
the two edges are oriented in the same direction. 
Then this implies that conditioned on $\{X_\star=X_+=0\}\cap\cF(\vN,\vM, n_+)$,
 we have $Z=Z_{00}+Z_{11}+Z_{10}.$
We begin by calculating the expectation of $Z_{00}$. Multiple edges of type
$00$ can only exist within $\hat\cN_0$, and the definition of $Z_{00}$ means we can
rule out any within $\setB$. 
Conditioned on $\hat\cF(\vN,\vM,n_+)$, for $v\in\hat\cN_0 $ we have that
$\hat d_{00}(v)$ is asymptotically distributed as $\Po(\hat m_{00}/\hat n_0)$.
Therefore,
\begin{align}\label{eq:Z00}
\Erw\brk{Z_{00} \; | \; X_\star = X_+=0,
\hat\cF(\vN,\vM,n_+)}& = (1+o(1))\Erw\brk{Z_{00}|\hat\cF(\vN,\vM,n_+)}
\nonumber\\
&= (1+o(1))\left(\binom{\hat n_0-\sizeB}{2}+\hat n_0 \sizeB\right)
\Erw\brk{\left.\binom{\hat d_{00}(v)}{2}\right| v\in \hat\cN_0}^2\frac{2}{(\hat m_{00}-1)(\hat m_{00}-3)}\nonumber\\
& =(1+o(1))\frac{(1-p)^2(1-\bar q)^2n^2 + 2(1-p)^2\bar qn^2}{2}\frac{(1-p)^4 d^4}{4} \frac{2}{(1-p)^4 d^2 n^2}\nonumber\\
& = (1+o(1))\frac{d^2}{4}(1-p)^2\left(1-\bar q^2 \right).
\end{align}
Similarly, multiple edges of type $11$ can only exist within $\hat \cN_1$. For $v\in\hat\cN_1$ we have that 
conditioned on $\hat\cF(\vN,\vM,n_+)$, 
$\hat d_{11}(v)$ is asymptotically distributed as 
$\Po_{\geq k}(\lambda_{11}).$ Therefore,
\begin{align}\label{eq:Z11}
\Erw\brk{Z_{11} \; | \; X_\star = X_+=0} & =(1+o(1))\Erw\brk{Z_{11}}
\nonumber\\
& = (1+o(1))\binom{\hat n_1}{2}\Erw\brk{\left.\binom{d_{11}(v)}{2}\right| v\in\hat\cN_1}^2 \frac{2}{(m_{11}-1)(m_{11}-3)}
\nonumber\\
&= (1+o(1))\frac{p^2(1-q)^2n^2}{2}\left(\frac{(dp)^2}{2p(1-q)}\pr\brk{\Po(dp)\ge k-2}\right)^2 \frac{2}{p^4d^2n^2}
\nonumber\\
& = (1+o(1))\frac{d^2}{4}\left(p+(1-p)\bar q\right)^2.
\end{align}
Finally we calculate the number of multiple edges of type $01/10$.
To this end, we aim to determine the asymptotic 
distribution of $\hat d_{10}(v)$ for $v\in \hat\cN_0.$
By \eqref{eq:qbarexp} we have 
\begin{align}\label{eq:bb10}
\hat n_0 \Erw \brk{\Po_{\le k-2}(\lambda_{10})}+\hat n_\star(k-1) & = \hat n_0 \sum_{x=0}^{k-2}x \frac{(dp)^x}{x!\exp(dp)(1-p)} + \hat n_\star (k-1)
\nonumber\\
& = (1+o(1))\left((1-p)n pd(1-\bar q) + pqn(k-1)\right) \nonumber\\
& = (1+o(1))\left((1-p)n pd(1-\bar q) + pn(1-p)\bar q d\right)
\nonumber\\
& = (1+o(1))p(1-p)dn = (1+o(1)) \hat m_{10}.
\end{align}
Conditioned on $\hat\cF(\vN,\vM,n_+)$, step (2) of
{\tt Forge} can be described by the following 
balls and bins experiment. Each of the $\hat m_{10}$ half-edges
of type $10$
are distributed uniformly at random over $\hat n_\star +\hat n_0$ vertices
subject to the condition that $\hat n_\star$ vertices
receive exactly $k-1$ and the remaining $\hat n_0$ vertices
all receive at most $k-2$. 
By \eqref{eq:bb10}, we have that 
$\Erw\brk{\Po_{\le k-2}(\lambda_{10})}\sim \hat m_{10}/\hat n_0.$
Since this is the distribution with highest entropy and this expectation, conditioned on $\hat\cF(\vN,\vM,n_+)$ for 
$ v\in \hat \cN_0$
we have that $\hat d_{10}(v)$ is asymptotically 
distributed as $\Po_{\leq k-2}(\lambda_{10}).$
For $v\in\hat\cN_\star\cup\hat\cN_1$, $\hat d_{01}(v)$
is asymptotically distributed as
$\Po(\hat m_{01}/(\hat n_\star+\hat n_1))$.
Therefore we have
\begin{align}\label{eq:Z01}
\hspace{2cm} & \hspace{-2cm} \Erw\brk{Z_{01}\; | \; X_\star=X_+=0,\hat\cF(\vN,\vM,n_+) }\nonumber\\
& = (1+o(1))\Erw\brk{Z_{01}|\hat\cF(\vN,\vM,n_+)}\nonumber\\
 &= (1+o(1))\left(\hat n_0 \Erw\brk{\left.\binom{\hat d_{10}(v)}{2}
 \right| v\in \hat \cN_0}
+ \hat n_\star\binom{k-1}{2}
\right)(\hat n_1+ \hat n_\star)\Erw\brk{\left.\binom{d_{01}(v)}{2}
\right| v\in \cN_\star\cup\hat\cN_1}
\frac{2}{\hat m_{01}(\hat m_{01}-1)}\nonumber\\
& = (1+o(1))\left((1-p)n \sum_{x=0}^{k-2}\frac{x(x-1)}{2}\frac{(dp)^x}{x!\exp(dp)(1-p)} + pqn\frac{k(k-1)}{2}\right)(\hat n_1 + \hat n_\star) \frac{\hat m_{01}^2}{2(\hat n_1 + \hat n_\star)^2}\frac{2}{\hat m_{01}^2}\nonumber\\
& = (1+o(1))\left(
(1-p)n\frac{(dp)^2}{2(1-p)}\pr\brk{\Po(dp)\le k-4}
+pqn\frac{(k-1)(k-2)}{2}
\right)
\frac{1}{pn}\nonumber\\
& = (1+o(1))\frac{d^2}{4}\left(2p\left(1-p-(1-p)\bar q-(1-p)\bar q \frac{k-2}{dp}\right)
+ \frac{2q}{d^2}(k-1)(k-2)
\right).
\end{align}
Summing~\eqref{eq:Z00},~\eqref{eq:Z11},~\eqref{eq:Z01}, and using \eqref{eq:qbarexp}, we obtain
\begin{align*}
\frac{\Erw\brk{Z \; | \; X_\star = X_+=0,\cF(\vN,\vM,n_+)}}{(1+o(1))d^2/4}
& = (1-p)\left((1-p)(1-\bar q^2) + (1-p) \bar q^2 + 2p\left(1-\bar q - \bar q \frac{k-2}{dp}\right)+ \frac{2}{d}\bar q (k-2) + 2p\bar q \right) + p^2\\
& = (1-p)\left(1+p +\bar q\left(- 2p -\frac{2}{d}(k-2) + \frac{2}{d}(k-2)+2p \right)\right)+p^2\\
& = (1-p)(1+p) + p^2 = 1.
\end{align*}
This completes the proof of the claim.
\end{proof}

We also need to estimate higher factorial moments, which correspond to the expected number of ordered tuples of cycles, loops or multiple edges. We will give the argument only for the higher moments of $X_\star$, since those of the other variables can be argued analogously.

So consider the expected number of ordered $r$-tuples of cycles of length $\ell_1,\ldots,\ell_r$ in $\cN_\star$. Recall that the expected number of cycles of length $\ell$ was asymptotically $\frac{1}{\ell}(k-1)q$. Thus the contribution made by $r$ pairwise disjoint cycles is asymptotically
$$
\prod_{i=1}^r \frac{1}{\ell_i} \left((k-1)q\right)^{\ell_i}.
$$
Summing over all choices of the $\ell_i$ we obtain
$$
\sum_{\ell_1,\ldots,\ell_r}\prod_{i=1}^r \frac{1}{\ell_i} \left((k-1)q\right)^{\ell_i}
= \prod_{i=1}^r \sum_{\ell_i}\frac{1}{\ell_i} \left((k-1)q\right)^{\ell_i} = \left(\Erw\brk{X_\star|\cF(\vN,\vM,n_+)}\right)^r.
$$

We would like to argue that the contribution made by tuples of cycles which are not pairwise disjoint is negligible. For this we prove a more general claim.

\begin{claim}\label{claim:nullity}
Let $\vN, \vM$ be such that $m_{11}$ is even and \eqref{eqthm:contig2}
holds. Then conditioned on $\hat\cF(\vN,\vM)$
\whp~ there are no sets of $s=O(1)$ vertices in $\hat \G$ which contain at least $s+1$ edges.
\end{claim}

\begin{proof}
We first crudely bound the degree distribution of any vertex of $\hat \G$ from above by $k-1 + \Po(d)$. Now given any pair of half-edges, the probability that they are matched is $O(1/n)$. Thus for a constant $s$, the expected number of sets of size $s$ containing at least $s+1$ edges is at most
$$
\binom{n}{s} \left((k-1+d)s\right)^{2s+2} O(1/n)^{s+1} = O(1/n).
$$
Thus by Markov's inequality, with high probability there is no such set, even taking a union bound over all $s=O(1)$.
\end{proof}

In particular, if an $r$-tuple of cycles is not pairwise disjoint, 
then it forms a subgraph with fewer vertices than edges.
By Claim~\ref{claim:nullity}, the contribution to the expected number of $r$-tuples of cycles made by those which are not pairwise disjoint is negligible.

This shows that
$$
\Erw \brk{X_\star^r|\cF(\vN,\vM,n_+)} =(1+o(1))\left(\Erw\brk{X_\star|\cF(\vN,\vM,n_+)}\right)^r 
$$
for any bounded $r$, and therefore 
by \Thm~\ref{Cor_mom}, $X_\star$ is asymptotically Poisson distributed with mean $\Erw \brk{X_\star}$. Therefore the probability that there is no directed cycle in $\cN_\star$ is asymptotically
$$
\exp(-\Erw\brk{X_\star|\cF(\vN,\vM,n_+)}) = (1+o(1))(1-(k-1)q).
$$
A similar argument works for each of the other expectations, and we obtain the results of \Prop~\ref{lem:condsprob}.

\subsection{Proof of Proposition~\ref{Prop_cycl}}
\label{sec:wellconstructed}
To prove \Prop~\ref{Prop_cycl} we will then show that $\hat\G$ is very unlikely to contain a flipping structure other than a forbidden cycle.
By \Prop~\ref{Prop_fs} (iv) any flipping structure that is not a forbidden cycle lies completely within 
	$\hat\cN_0\times\hat\cN_0$ and contains at least one vertex from 
	$\hat\cN_0\setminus\hat\cN_+$.
The following two lemmas establish that given $\cE\cap\hat\cF(\vN,\vM)$,
there are no such flipping  structures \whp\
We consider two cases separately, depending on the the {order} of the
flipping structure, i.e., the number of vertices in $\Sundir$.

\begin{lemma}\label{claim_fssmall}
There exists $\eps_1=\eps_1(d,k)>0$ such that the following is true.
Let $\vN,\vM$ be such  that $m_{11}$ is even and \eqref{eqthm:contig2} holds.
Then conditioned on $\hat\cF(\vN,\vM)\cap\cE$ \whp\ there is no
flipping structure of order 
at most $\eps_1 n$  in $\hat\cN_0\times \hat\cN_0$
that contains at least one vertex from $\hat\cN_0\setminus \hat\cN_+.$
\end{lemma}

\begin{lemma}\label{claim_flip}
Let $\vN,\vM$ be such  that $m_{11}$ is even and \eqref{eqthm:contig2}
holds.
Then conditioned on $\hat\cF(\vN,\vM)\cap\cE$
\whp\ there are no flipping structures of order at least 
$\eps_1 n$ in $\hat \cN_0$.
\end{lemma}

\noindent
We prove \Lem s~\ref{claim_fssmall} and~\ref{claim_flip} in \Sec s~\ref{Sec_claim_fssmall} and~\ref{Sec_claim_flip}.
But let us first point out that \Prop~\ref{Prop_cycl} is an immediate consequence of \Prop~\ref{Prop_fs} and 
\Lem s~\ref{claim_fssmall} and~\ref{claim_flip}.

\begin{proof}[Proof of \Prop~\ref{Prop_cycl}]
We have
\begin{align*}
\pr\brk{\cF(\vN,\vM)|\hat\cF(\vN,\vM)}&=
\pr\brk{\cF(\vN,\vM)|\cE\cap\hat\cF(\vN,\vM)}
\pr\brk{\cE\cap\hat\cF(\vN,\vM)|\hat\cF(\vN,\vM)}\\
&=
\pr\brk{\cF(\vN,\vM)|\cE\cap\hat\cF(\vN,\vM)}
\pr\brk{\cE|\hat\cF(\vN,\vM)}.
\end{align*}
That is, our aim is 
to show that 
$\pr[\cF(\vN,\vM)|\cE\cap\hat\cF(\vN,\vM)]=1+o(1).$
Certainly, given $\cE$ it holds that 
$\hat\G$ is simple. 
Further, given $\cE_2$, 
Proposition~\ref{Prop_fs}~\eqref{claim_fs1} and~\eqref{claim_fs2}
imply that a possible 
flipping structure must lie completely 
within $\cN_0.$
Similarly, given $\cE_3$, Proposition~\ref{Prop_fs}~\eqref{claim_fs3}
implies that there is no flipping
structure completely within $\hat\cN_+.$
Therefore  invoking {\Lem s~\ref{claim_fssmall} and~\ref{claim_flip}} we conclude that given
$\cE\cap\hat\cF(\vN,\vM)$ \whp~
$\cF(\vN,\vM)$ holds, as required.
\end{proof}

\subsubsection{Proof of \Lem~\ref{claim_fssmall}}\label{Sec_claim_fssmall}
Let $\hat\cN_-=\hat\cN_0\setminus\setB$ and for a set $S\subset[n]^2$ let $V_-(S)={V(S)}\cap\hat\cN_-$ and $V_+(S)={V(S)}\cap\hat\cN_+$.
Further, denote by $G_-(S)$ and $G_+(S)$ the subgraphs of $\Sundir$ induced on $V_-(S)$ and $V_+(S)$ respectively.
Additionally, let $a=a(S)=|V_-(S)|$ and $b=b(S)=|V_+(S)|$ and let $i=i(S)$ be the number of vertices that are isolated in $G_+(S)$.
We assume throughout that
	\begin{equation*}
	a+b\leq\eps_1 n.
	\end{equation*}
Let $\ell=\ell(S)$ be the number of leaves (i.e., vertices of degree one) in $G_+(S)$. 
{Let $c=c(S)$ denote the number of  components of order at least two in $G_+(S)$. Let $x=x(S)$ denote the number of edges  in $G_-(S).$}
Throughout this section we assume that 
$0<\eps_1 \ll \eps_2 \ll \eps_3 \ll \eps_4(d,k)$.

\begin{fact}\label{Fact_basic}
Given that $\hat\G$ is simple, the following statements hold for any flipping structure $S\subset\hat\cN_0\times\hat\cN_0$ with $V(S)\cap\hat\cN_-\neq\emptyset$.
\begin{enumerate}
\item $G_+(S)$ is acyclic.
\item 
Every leaf of $G_+(S)$ has a $G(S)$-neighbour in $V_-(S)$. 
\item 
Every isolated vertex of $G_+(S)$ has at least two $G(S)$-neighbours in $V_-(S)$.
\item Every vertex in $G_-(S)$ has at least three $G(S)$-neighbours.
\end{enumerate}
\end{fact}
\begin{proof}
If $G_+(S)$ contains a cycle, then this cycle is itself a flipping structure, and thus $S$ is not minimal.
This shows (1) and (2), (3) follow from  \Prop~\ref{Prop_fs} (viii).
Finally, (4) follows from \Prop~\ref{Prop_fs} (vii).
\end{proof}

\begin{claim}\label{Claim_claim_fssmall_1}
Let $\vN,\vM$ be such that $m_{11}$ is even and 
\eqref{eqthm:contig2} holds.
Given
$\hat\cF(\vN,\vM)\cap\cE_1$,
\whp\ $\hat\G$ does not contain a flipping structure $S\subset\hat\cN_0\times \hat\cN_0$
with $
a+b\leq\eps_1 n$ such that $a\geq\eps_2 b$.
\end{claim}
\begin{proof}
Fact~\ref{Fact_basic} implies that  the induced subgraph 
$\Sundir$ of $\hat\G[\hat\cN_0]$ has average degree
at least 
	$$\frac{3a + 2b}{a+b}\ge
	2+\frac{\eps_2}{2}.$$
But by \Cor~\ref{claim:avdeg},   for $\eps_1=
\eps_1(\eps_2,d,k)>0$ small enough $\hat\G[\hat\cN_0]$ does not contain such a subgraph \whp\
\end{proof}

\begin{claim}\label{Claim_claim_fssmall_4}
Let $\vN,\vM$ be such that $m_{11}$ is even and 
\eqref{eqthm:contig2} holds.
Given $\hat\cF(\vN,\vM)\cap\cE_1$, \whp\ $\hat\G$ does not contain a flipping structure $S\subset\hat\cN_0\times \hat\cN_0$
with $a+b\leq\eps_1 n$ such that $a\leq\eps_2b$ and $i\geq\eps_3b$.
\end{claim}
\begin{proof}
Every isolated vertex of $G_+(S)$ has at least two $G(S)$-neighbours in $V_-(S)$.
Therefore, \Cor~\ref{claim:avdeg2} applies.
\end{proof}

\begin{claim}\label{Claim_claim_fssmall_2}
Let $\vN,\vM$ be such that $m_{11}$ is even and 
\eqref{eqthm:contig2} holds.
Given $\hat\cF(\vN,\vM)\cap\cE_1$,
\whp\ $\hat\G$ does not contain a flipping structure $S\subset\hat\cN_0\times \hat\cN_0$
with $a+b\leq\eps_1 n$ such that $a\leq\eps_2b$,  $i\leq\eps_3b$ and $\ell\geq\eps_4(b-i)$.
\end{claim}
\begin{proof}
We aim to determine the average degree in the induced subgraph 
$G(S)$ of $\hat\G[\hat\cN_0].$
By \Prop~\ref{Prop_fs}~\eqref{claim_fsnew} each 
vertex
in $G(S)$   
has degree at least $2$ in $G(S)$.
That is, the total
degree among the vertices of $G(S)$ in $\hat\cN_+$ is at least $2b.$
It remains to determine the total degree among the vertices 
 of $G(S)$ in $\hat\cN_-$. 
 By Fact~\ref{Fact_basic}~
every leaf of $G_+(S)$ has a 
$G(S)-$neighbour  in $V_-(S)$, and
 each isolated vertex in $G_+(S)$ has at least
 two $G(S)$-neighbours in $V_-(S)$. That is, there are at least 
 $2i +\ell$ edges between $V_+(S)$ and $V_-(S)$  in $G(S)$ and so the total
 degree among the vertices of $G(S)$ in $\hat\cN_-$ is at least
  $2i+\ell.$ Since
  $\ell\geq\eps_4(1-\eps_3)b$ and $a\leq \eps_2 b$,
  the average degree in $G(S)$ is at least
  $$
  \frac{2b+2i+\ell}{a+b}\geq\frac{2b+\ell}{a+b}
  \geq \frac{2+\eps_4(1-\eps_3)}{1+\eps_2}.
  $$
 But by \Cor~\ref{claim:avdeg},   for $\eps_3<1$,
 $\eps_2=\eps_2( d,k,\eps_3, \eps_4)$ and
 $\eps_1=
\eps_1(d,k,\eps_2)>0$ small enough $\hat\G[\hat\cN_0]$ does not contain such a subset \whp\
\end{proof}

\begin{claim}\label{Claim_claim_fssmall_x}
Let $\vN,\vM$ be such that $m_{11}$ is even and 
\eqref{eqthm:contig2} holds.
Given $\hat\cF(\vN,\vM)\cap\cE_1$
 \whp~
$\hat\G$ does not contain a flipping structure $s\subset\cN_0\times\hat\cN_0$ with 
$a+b\le\eps_1 n$ and  $x\geq 1.01 a$.
\end{claim}
\begin{proof}
If $x\geq 1.01 a$, the induced subgraph $G_-(S)$ has average degree
$2.02.$ By \Cor~\ref{claim:avdeg} for $\eps_1=\eps(0.01,d,k)$
no such subgraph exists in $\hat\G[\hat \cN_0].$
\end{proof}

\begin{claim}\label{Claim_claim_fssmall_smalli}
Let $\vN,\vM$ be such that $m_{11}$ is even and 
\eqref{eqthm:contig2} holds.
Given
$\hat\cF(\vN,\vM)\cap\cE_1$
\whp\ $\hat\G$ 
does not contain a flipping structure $S\subset\hat\cN_0\times \hat\cN_0$ with the following properties.
	\begin{enumerate}[(1)]
	\item $a+b\le\eps_1 n$,
	\item $x> 0.99 a$,
	\item $i \geq 0.1a$.
	\end{enumerate}
\end{claim}
\begin{proof}
We aim to determine the average degree in the induced subgraph of $G(S)$
on $V_-(S)$ and the isolated vertices of $G_+(S)$.
By Fact~\ref{Fact_basic}~
every  isolated vertex in $G_+(S)$ has at least
 two $G(S)$-neighbours in $V_-(S)$. By assumption there are $x$
 edges in $G_-(S).$ 
 Therefore the average degree is
 \begin{align*}
 \frac{2x +4 i}{a+i}\geq \frac{1.98 a+4i }{a+i}\geq 
 \frac{2.38}{1.1}.
 \end{align*}
 By \Cor~\ref{claim:avdeg} for $\eps_1=\eps_1(d,k)$
no such subgraph exists in $\hat\G[\hat \cN_0].$
\end{proof}

\begin{claim}\label{Claim_claim_fssmall_3}
Let $\vN,\vM$ be such that $m_{11}$ is even and 
\eqref{eqthm:contig2} holds.
Given $\hat\cF(\vN,\vM)\cap \cE_1$
\whp\ $\hat\G$ 
does not contain a flipping structure $S\subset\hat\cN_0\times \hat\cN_0$ with the following properties.
	\begin{enumerate}[(1)]
	\item $a+b\le\eps_1 n$,
	\item $x>0.99a$,
	\item $\ell-c+i\leq a\leq\frac{100}{99}(c+i)$.
	\end{enumerate}
\end{claim}
\begin{proof}
By Claim~\ref{Claim_claim_fssmall_smalli} 
\whp~ there are no 
flipping structures with $x>0.99a$ and  $i\geq 0.1a$.
Now, assume that there is a flipping structure $S$
  with~(3) and 
$i\leq0.1a$. 
For such a flipping structure, from the assumption that $\ell-c+i \leq a \leq \frac{100}{99}(c+i)$ 
and $c\leq \ell/2$ we obtain that $c\geq \ell /2.25$.
Each component in $S$
 that is not an isolated vertex has at least two leaves.
Therefore, letting $c'=c'(S)$ be the number of components of order at least $2$ in $S$ with exactly two leaves, we conclude that
$\ell \geq 2c' + 3(c-c')$, and thus $c'\geq0.75c$.
This implies that there are at least $c'$ paths contained in 
$\hat\cN_+$ whose endpoints are adjacent to vertices in $V_-(S)$.
Consequently, \Cor~\ref{Claim_claim_fssmall_path}
completes the proof.
\end{proof}

The rest of the proof is based on the first moment method.
Let 
$\nu_+=(1-p)\bar q$ and $\nu_-=(1-p)(1-\bar q)$
and pick a slowly growing $\omega=\omega(n)\to\infty$.
By Claim~\ref{claim_nplus} and
\Prop~\ref{lem:condsprob},
because $\hat n_++\hat n_-=\hat n_0$ we have 
 	\begin{equation}
	\pr\brk{|\sizeB-\nu_+ n |+ |\hat n_--\nu_- n |\le 3\omega\sqrt{n}|\hat \cF(\vN,\vM)\cap \cE}=1-o(1).\label{eq:approxn+-}
		\end{equation}
Let $\cA(3\omega)$ denote the event that 
$|\sizeB-\nu_+ n |+ |\hat n_--\nu_- n |\le 3\omega\sqrt{n}$ holds.
By Claim~\ref{claim_nplus1}, the number {$\hat m_+$} of edges within 
$\hat\G[\setB]$ is distributed as $\Bin(\hat m_{00}/2,(\sizeB/\hat n_0)^2)$.
Hence, setting $\hat \cB=\hat\cF(\vN,\vM)\cap\cE_1\cap\hat\cA(3\omega)$
$$\Erw[\hat m_+
|\hat\cB]=
\frac{\hat m_{00}}{2}\left(\frac{\sizeB}{\hat n_0}\right)^2 \sim\frac{\gamma_+}{2}\sizeB.$$
Further, a Chernoff bound implies 
that conditioned on 
$\hat\cB$
\whp~we have 
 	\begin{equation}
	|2\hat m_+-\gamma_+\hat n_+|\le \omega\sqrt{n}.\label{eq:approxm+}
		\end{equation}

We begin with deriving an auxiliary proposition bounding the following quantity, which will appear in the rest of the proof.
Let
	\begin{align*}
	C=C(a,b,c,\ell,i)
	&=\bink{\hat n_+}{b-i}\bink{\hat n_+-(b-i)}i\bink{\hat n_-}a\bcfr{2\hat m_+}{\hat n_+^2}^{b-i-c}
		\frac{(b-i)!}{\ell!}\bink{b-i-1}{c-1}\stir(b-i-c,b-i-\ell).
	\end{align*}
Let 
\begin{align*}
B=B(a,b,c,i)=
	\bink{\hat n_-}{a}\bink{\hat n_+-(b-i)}{i}
	\br{\frac{\hat n_+}{c}}^c
	\le
	\eul^{a+i}\br{\frac{n}a}^a\br{\frac ni}^i
\br{\frac{n}{c}}^c
\end{align*}
and let $f(x)=-x\ln(x).$

\begin{proposition}
If $c \le \eps_4 b$, then conditioned on $\hat \cB$
we have
	\begin{align}
	C
	&\leq\eps_4 B
\gamma_+^{b-c-i}	\sqrt{\frac{b-i}{\ell}}
	\exp\brk{2\ell+b\br{
	f\br{\frac {\ell}{b-i}}
	+2f\br{\frac{\ell}{2(b-i)}}}+O\br{\frac{\omega b}{\sqrt n}}
	+(c-1)\ln\br{\frac{c}{c-1}}}.\label{eq:boundC}
	\end{align}
\end{proposition}

\begin{proof}
Using \Thm~\ref{thm_stir} and upper bounding 
$$
\frac{(b-i)!}{\ell!}\leq \eul^{\ell-b+i+1}\frac{(b-i)^{b-i+1/2}}{\ell^{\ell+1/2}},
$$
we obtain 
\begin{align*}
C& \leq 
\eul^{\ell-b+i+1}\sqrt{\frac{b-i}{\ell}}
\bink{\hat n_+}{b-i}\bink{\hat n_+ -(b-i)}i\bink{\hat n_-}a
\bink{b-i-1}{c-1}\bink{b-i-c}{\ell-c}
\frac{(b-i)^{b-i}(b-i-\ell)^{\ell-c}}{\ell^\ell}
\bcfr{2\hat m_+}{\hat n_+^2}^{b-i-c}.
\end{align*}

From \eqref{eq:approxn+-} and ~\eqref{eq:approxm+} we obtain that conditioned on $\hat\cB$
we have $(2\hat m_+/\gamma_+\hat n_+)^{b-i-c}\leq\exp(O(\omega b/\sqrt n))$   and 
$(\hat n_+/\nu_+ n)^{2b-2i}\leq\exp(O(\omega b/\sqrt n))$.
Therefore, conditioned on $\hat\cB$
\begin{align*}
C
& \leq \eul^{\ell-b+i+1} \sqrt{\frac{b-i}{\ell}}\left(\frac{\eul\hat n_+}{b-i}\right)^{b-i}B\left(\frac{c}{\hat n_+}\right)^c \left(\frac{(b-i)\eul}{c-1}\right)^{c-1} \left(\frac{(b-i)\eul}{\ell-c}\right)^{\ell-c} \frac{(b-i)^{b-i}(b-i)^{\ell-c}}{\ell^\ell} \left(\frac{\gamma_+}{\hat n_+}\right)^{b-i-c}
\exp\br{O\br{\frac{\omega b}{\sqrt n}}}\\
& \leq \eul^{\ell-b+i+1} \sqrt{\frac{b-i}{\ell}}\left(\frac{\eul\nu_+ n}{b-i}\right)^{b-i}B\left(\frac{c}{\nu_+ n}\right)^c \left(\frac{(b-i)\eul}{c-1}\right)^{c-1} \left(\frac{(b-i)\eul}{\ell-c}\right)^{\ell-c} \frac{(b-i)^{b-i}(b-i)^{\ell-c}}{\ell^\ell} \left(\frac{\gamma_+}{\nu_+ n}\right)^{b-i-c}
\exp\br{O\br{\frac{\omega b}{\sqrt n}}}\\
&  \leq \eul^{2\ell} \sqrt{\frac{b-i}{\ell}}B
\br{\frac{c}{c-1}}^{c-1}
	\gamma_+^{b-c-i}\frac{c}{b-i}
	\br{\frac{b-i}{\ell}}^{\ell}
	\br{\frac{b-i}{\ell-c}}^{\ell -c}
	\exp\br{O\br{\frac{\omega b}{\sqrt n}}}. 
\end{align*}
The bound on $C$ follows directly from the assumption that $c \le \eps_4 b$.
\end{proof}

\begin{claim}\label{Claim_claim_fssmall_Olly1}
Let $\vN,\vM$ be such that $m_{11}$ is even and 
\eqref{eqthm:contig2} holds.
Given $\hat\cF(\vN,\vM)\cap \cE_1$
\whp\ $\hat\G$ 
does not contain a flipping structure $S\subset\hat\cN_0\times \hat\cN_0$ with the following properties.
	\begin{enumerate}[(1)]
	\item $a+b\le\eps_1 n$,
	\item $a\leq\eps_2b$,
	\item $i\leq\eps_3b$,
	\item $\ell\leq\eps_4(b-i)$,
	\item $a <\ell-c+i$
	\end{enumerate}
\end{claim}

\begin{proof}

Let $Z'$ denote the number of such flipping structures $S$.
Recall that
each leaf of 
$G_+(S)$
 must have a $G(S)$-neighbour among the $a$ vertices in $V_-(S)$, and every isolated vertex must have two $G(S)$-neighbours in $V_-(S)$.
 Since the existence of these edges is a monotone graph property, by 
\Lem~\ref{lem_Nzero} the probability that all necessary edges
are present is upper bounded up to a constant by 
		$$R'=R'(a,\ell,i)=\bcfr{a\gamma_0 }{\hat n_0}^{\ell+2i}.$$
Therefore
\begin{align}\label{eq:expZ'}
\Erw[Z'|\hat  \cB]\leq O(
C\cdot R'
).
\end{align}
Conditioned on $\hat\cB$ we have 
$$
R'
\leq\bcfr{a\gamma_0 }{\nu_0 n}^{\ell+2i}
\exp\br{O\br{\frac{b}{\sqrt n}}}.$$
Since $a\leq b$, we obtain that conditioned on $\hat\cB$
	\begin{align}
B\cdot R'
&\leq
\eul^{a+i}\br{\frac{\gamma_0}{\nu_0}}^{\ell+2i}
\br{\frac{a}{i}}^i
\br{\frac{a}{c}}^c
\br{\frac an}^{-a+\ell+i-c}
\exp\br{O\br{\frac{b}{\sqrt n}}}
\nonumber\\&
\leq
\br{\frac{\gamma_0}{\nu_0}}^{\ell+2i}
\br{\frac an}^{-a+\ell+i-c}
\exp\brk{a+i+b\br{f\br{\frac cb}+f\br{\frac ib}} +O\br{\frac{b}{\sqrt n}}}
.\label{eq:boundBR}
	\end{align}
The map $f$ is continuous and monotonically
increasing on $[0,1/\eul)$ with $f(x)\to 0$ as $x\to 0$.
Therefore using $a\leq\eps_2 b,$
$i\leq \eps_3b$, $c\leq \ell /2
\leq \eps_4(1-\eps_3)b/2$
and $a\leq \ell-c+i,$
from \eqref{eq:boundC} and \eqref{eq:boundBR} we obtain
that
for $0<\eps_4<1$, $\eps_3=\eps_3(d,k,\eps_4)$, $\eps_2=\eps_2(d,k,\eps_3)>0$ small enough it holds that
\begin{align}\label{eq:case1}
	C\cdot R'
	& \stackrel{\eqref{eq:boundC}}{\le} 
	B\cdot R'\cdot \eps_4  \gamma_+^{b-c-i}	\sqrt{\frac{b-i}{\ell}}
	\exp\brk{2\ell+b\br{
	f\br{\frac {\ell}{b-i}}
	+2f\br{\frac{\ell}{2(b-i)}}}+O\br{\frac{\omega b}{\sqrt n}}+(c-1)\ln\br{\frac{c}{c-1}}}\nonumber\\
	&\stackrel{\eqref{eq:boundBR}}{\le}
	\br{\frac{\gamma_0}{\nu_0}}^{\ell+2i}
	\br{\frac an}^{-a+\ell+i-c}
	\gamma_+^{b-c-i}	\sqrt{\frac{b-i}{\ell}}\nonumber\\
	& \hspace{0.2cm}\cdot \exp\brk{2\ell+a+i+b\br{
	f\br{\frac {\ell}{b-i}}
	+2f\br{\frac{\ell}{2(b-i)}}+f\br{\frac cb}+f\br{\frac ib}}
	+(c-1)\ln\br{\frac{c}{c-1}}+O\br{\frac{\omega b}{\sqrt n}}}\nonumber\\
	& \le 	
	\br{\frac{\gamma_0}{\nu_0}}^{(2\eps_2+\eps_4)b}
	\gamma_+^{b}	\sqrt{b}
	\exp\brk{2b(\eps_3+\eps_4)+b\br{
	3f\br{2\eps_4}
	+f\br{\eps_4}+f\br{\eps_2}}+\eps_4b+O\br{\frac{\omega b}{\sqrt n}}},
	\end{align}
	where the last line follows since $\gamma_+<1$.
	For $\omega\to\infty$ slowly enough by
~\eqref{eq:expZ'} and~\eqref{eq:case1} we obtain 
\begin{align}
\Erw[Z'|\hat\cB]=O(
C\cdot R'
)=o(1)
\end{align}
as required.
\end{proof}

\begin{claim}\label{Claim_claim_fssmall_Olly2}
Let $\vN,\vM$ be such that $m_{11}$ is even and 
\eqref{eqthm:contig2} holds.
Given $\hat\cF(\vN,\vM)\cap \cE_1$
\whp\ $\hat\G$ 
does not contain a flipping structure $S\subset\hat\cN_0\times \hat\cN_0$ with the following properties.
	\begin{enumerate}[(1)]
	\item $a+b\le\eps_1 n$,
	\item $a\leq\eps_2b$,
	\item $i\leq\eps_3b$,
	\item $\ell\leq\eps_4(b-i)$,
	\item $x \le 1.01a$,
	\item $a\geq \ell-c+i$.
	\item Either $a \ge \frac{100}{99}(c+i)$ or $x \le 0.99a$.
	\end{enumerate}
\end{claim}

\begin{proof}

Recall that in such a flipping structure $S$,
every vertex in $G_-(S)$ must have at least three neighbours
in $G(S)$.
	Since $x$ is the number of edges within $G_-(S)$, there must be $3a-2x$ other edges and we obtain  the probability that 
	all necessary edges are present is bounded up to a constant by
			\begin{align*}
		R''=R''(a,b,x)
		&=\bink{\bink a2}{x}\bink{ab}{3a-2x}\bcfr{\gamma_0}{
		\hat n_0}^{3a-x}.
		 \end{align*}
Conditioned on $\hat\cB$, from \eqref{thm:contig2} we obtain
		 \begin{align*}
		 R''
		 &\leq \bink{\bink a2}{x}\bink{ab}{3a-2x}\bcfr{\gamma_0}{
		\nu_0 n}^{3a-x}\exp\br{O\br{\frac{b}{\sqrt n}}}\\
		&\leq\eul^{3a}\br{\frac{\gamma_0}{\nu_0}}^{3a-x}
		\br{\frac{a}{n}}^{3a}
		\br{\frac{n}{x}}^x
		\bcfr{b}{3a-2x}^{3a-2x}\exp\br{O\br{\frac{b}{\sqrt n}}}.
				\end{align*}
	Hence, for $a\leq b$,
	\begin{align}
	&B\cdot R''
	\leq \eul^{4a+i}
	\br{\frac{\gamma_0}{\nu_0}}^{3a-x}
	\br{\frac ax}^x
	\br{\frac{a}{i}}^i
	\br{\frac{a}{c-1}}^c
	\br{\frac{b}{3a-2x}}^{3a-2x}
	\br{\frac an}^{2a-x-c-i}
	\exp\br{O\br{\frac{b}{\sqrt n}}}\nonumber\\
&\leq 
\br{\frac{\gamma_0}{\nu_0}}^{3a-x}
\br{\frac an}^{2a-x-c-i}
\exp\brk{
4a+i+b
\br{
f\br{\frac xb} + f\br{\frac ib}+f\br{\frac cb}+
f\br{\frac{3a-2x}{b}}}
+(c-1)\ln\br{\frac{c}{c-1}}+\br{O\br{\frac{b}{\sqrt n}}}}.
		\label{eq:boundBR2}
	\end{align}
\begin{description}
\item[Case 1: $a>(100/99)(c+i)$. ]
Let $Z''(a,b,c,i,\ell,x)$ be the number of
flipping structures satisfying the conditions of the Claim and also $a>(100/99)(c+i)$, which implies
$2a-x-c-i>0$.
From 
\eqref{eq:boundC} and \eqref{eq:boundBR2} we obtain
that
for $\eps_4>0$, $\eps_3=\eps_3(d,k,\eps_4)$, $\eps_2=\eps_2(d,k,\eps_3)>0$ small enough
\begin{align}\label{eq:expZ''}
\Erw[Z''|\hat\cB]=O(
C\cdot R''
)=o(1).
\end{align}

\item[Case 2:
$\ell-c+i\leq a \leq (100/99)(c+i),
x\leq 0.99a$. ]
Finally, denote by $Z'''(a,b,c,i,\ell,x)$  the number of
flipping structures satisfying the conditions of the claim
and  $\ell-c+i\leq a \leq 100/99(c+i), x\leq 0.99 a$.
Again we obtain $2a-x-c-i>0$ and 
\begin{align}\label{eq:expZ'''}
\Erw[Z'''|\hat\cB]=O(
C\cdot R''
)=o(1).
\end{align}
\end{description}
The assertion follows from combining~\eqref{eq:expZ''} and~\eqref{eq:expZ'''}.
\end{proof}

\begin{proof}[Proof of \Lem~\ref{claim_fssmall}]
From Claims~\ref{Claim_claim_fssmall_1}
--\ref{Claim_claim_fssmall_3}
and Claims~\ref{Claim_claim_fssmall_Olly1} and~\ref{Claim_claim_fssmall_Olly2}
 we obtain that 
conditioned of $\cF(\vN,\vM)\cap \cE_1$ \whp~ there
is no flipping structure of order at most 
$\eps_1n$. The assertion follows since from \Prop~\ref{lem:condsprob}
we have $\pr\br{\cE_2\cap\cE_3|\cF(\vN,\vM)\cap \cE_1}=\Theta(1)$.
\end{proof}

\subsubsection{Proof of \Lem~\ref{claim_flip}}\label{Sec_claim_flip}

Assume that there is a flipping structure on at least $\eps_1n$
vertices of $\hat\cN_0$, then by \Prop~\ref{Prop_fs}~ 
\eqref{claim_fs-3} for
every pair of vertices $(v,w)$ in $S$  
we have that $
\mu_{v\to w}(\hat\G)\geq \mu_{v\to w}(\hat\G,S)=1.$ 
That is, there has to be a set of $\eps_1 n$ vertices $v\in \hat \cN_0$ such that applying Warning Propagation
on $\hat\G$ would result in a message of type $\mu_{v\to w}(\hat\G)=1$, whereas $\hat\vmu_{v\to w}=0$.
 
We aim to show that given $\cF(\vN,\vM)\cap \cE_1$ \whp~such a set does not
exist in $\hat\cN_0$ by exploring the component of $v\in\hat\cN_0$
in $\hat\G$
and describing the local neighbourhood of $v$ by a two-type 
branching process. By construction $v$ can have neighbours 
 incident to half-edges of type $00$ and $10$ only. 
Further conditioned on $\cE_1,$ for each half-edge of type $00$ the 
matching in step (5) of {\tt Forge} will result in an edge from $v$ to another vertex $w\in\hat\cN_0$. Similarly,
each half-edge of type $10$ the 
matching will result in an edge from $v$ to vertex $w\in\hat\cN_\star\cup\hat\cN_1$. 

Conditioned on $\cF(\vN,\vM),$ the number $X$ of neighbours of $v$ in $\hat \cN_\star\cup \hat \cN_1$
is asymptotically distributed as $\Po_{\leq k-2}(dp)$,
and the number $Y$ of neighbours in $\hat\cN_0$ is asymptotically distributed as
$\Po(d(1-p))$ independently of $X$. We define a $2$-type branching process with these parameters,
i.e. we start from a vertex $v$ of type $\hat \cN_0$ and each vertex of type $\hat \cN_0$
has $\Po_{\le k-2}(dp)$ children of type $\hat \cN_\star \cup \hat \cN_1$ and $\Po(d(1-p))$ children of type $\hat \cN_0$ independently. 
Vertices of type $\hat \cN_\star \cup \hat \cN_1$ have no children in this branching process.

To prove \Lem~\ref{claim_flip} we show that applying Warning 
Propagation to this branching process would result in a message of type $1$ at $v$. We may
assume that a child of $v$ in $\hat \cN_\star \cup \hat \cN_1$ will always send message $1$ towards
$v$ in the tree. This is necessary because we ignored any children of such vertices.
Let 
$Y_t$ be the number of children of $v$ in $\hat\cN_0$ that send a $1$ towards $v$ after $t$ iterations of Warning Propagation.

Now, let $u_t=\vecone\cbc{X+Y_t\geq k-1}$ .
Our aim is to bound $\pr\brk{u_t=1}$ from above.
By the recursive structure of the tree, $Y_t$ has $\Po(d(1-p)\Erw u_{t-1})$ distribution independently
of $X$. Now, recall \eqref{eq:varphi}.
Setting $\bar u_t=\Erw u_t$, by the assumptions that we made it holds that 
\begin{align*}
 \Erw\brk{u_t} 
  \le \pr\brk{X+Y_t\geq k-1}
  &=\sum_{j=0}^{k-2}\frac{(dp)^j}{(1-p)j!\exp(dp)}\pr\brk{\Po(d(1-p)\bar u_{t-1})\geq k-1-j}\\
  &\stackrel{\eqref{eq:varphi}}{=}\sum_{j=0}^{k-2}\frac{(dp)^j}{(1-p)j!\exp(dp)}\varphi_{k-j}(d(1-p)\bar u_{t-1})
  =:f_k(\bar u_{t-1}).              
\end{align*}
We will prove that $f_k(x)<x$ for all $x\in (0,1]$
by showing that $f_k$ has derivative strictly less than $1$ 
on $(0,1]$.
By definition, $f_k(x)\geq 0$ with equality iff $x=0$, and
\begin{equation*}
f_k(1) \le \varphi_2(d(1-p)) = 1-\exp(-d(1-p)) <1.
\end{equation*}

Using \eqref{eq:deriv} we obtain
\begin{align}
 \frac{\partial}{\partial x} f_k(x) 
  &= \frac{d(1-p)}{1-p} \sum_{j=0}^{k-2}\frac{(dp)^j}{j!\exp(dp)}
  \frac{(d(1-p)x)^{k-2-j}}{(k-2-j)!\exp(d(1-p)x)}\nonumber\\
  &=d\pr\brk{\Po(dp) + \Po(d(1-p)x) = k-2}
  =d\frac{\partial}{\partial y}\varphi_{k}(y)|_{y=d(p+(1-p)x)}
  \label{eq:derivf}
\end{align}
and therefore
\begin{align*}
\frac{\partial^2}{\partial x^2} f_k(x)=d^2(1-p) 
\frac{\partial^2}{\partial y^2}\varphi_{k}(y)|_{y=d(p+(1-p)x)}.
\end{align*}
Since 
$\frac{\partial}{\partial y}\varphi_k(y)$
is positive for $y\geq0$, so is 
$\frac{\partial}{\partial x}f_k(x)$ for 
all $x\in [0,\infty)$,
i.e.\ $f_k$ is monotonically increasing
on $[0,\infty)$.
Similarly since 
$$\sign\br{\frac{\partial^2}{\partial y^2}\varphi_k(y)}\stackrel{\eqref{eq:deriv}}{=}\sign (k-2 -y),$$
we have that 
$\frac{\partial^2}{\partial x^2}f_k(x)\leq 0$ for 
all $x\geq (k-2-dp)/(d(1-p))\cap 0$.
By Fact~\ref{fact_FPana} (1) we have that 
$dp\geq k-2$
i.e.\ $f_k$ is concave on the entire interval 
$[0,\infty)$. 

Recalling the definition of $\phi_{d,k}$
in \eqref{eqmain}, 
we obtain that 
$\frac{\partial}{\partial x}\phi_{d,k}(x)|_{x=p}
=d\frac{\partial}{\partial y}\varphi_k(y)|_{y=dp}$. Therefore
\eqref{eq:derivf} implies that
\begin{align*}
\frac{\partial}{\partial x}f_k(x)|_{x=0}
=\frac{\partial}{\partial x}\phi_{d,k}(x)|_{x=p}.
\end{align*}
Hence, by Fact~\ref{fact_FPana} (2) we obtain that 
$\frac{\partial}{\partial x} f_k(x)|_{x=0}<1.$ Since $f_k$ is 
monotonically increasing and concave on $[0,\infty)$ this 
implies that $f_k$ has derivative
strictly less than one on $[0,\infty)$
and therefore $f_k(x)<x$ for all $x>0.$
 
We may thus conclude that $0$ is the only non-negative fixed point of the function $f_k$, and therefore $\bar u_t \to 0$.
Thus also $u_t\to 0$ \whp~ In other words, each vertex has probability $o(1)$
of lying in any flipping structure.
Thus the expected number of vertices in any flipping structure is $o(n)$ and by Markov's inequality, conditioned on $\cF(\vN,\vM)\cap \cE_1$ \whp\ there is certainly no flipping structure of order at least $\eps_1 n$.
Again the result follows since by \Prop~\ref{lem:condsprob} we have
$\pr[\cE_2\cap \cE_3|\hat\cF(\vN,\vM)\cap \cE_1]=\Theta(1).$

\section{Proof of \Prop~\ref{Prop_entropy}}\label{section_LLTproofs}

\noindent{\em We keep the notation and assumptions from \Prop~\ref{Prop_entropy}}

\medskip\noindent
In light of \Prop~\ref{Prop_uniform}  we basically need to study  the entropy of the output distribution of {\tt Forge} given $\cF(\vN,\vM)$.
Given $\vN=(n_\star, n_1)$, $\vM=(m_{10},m_{11})$ let 
	\begin{align*}
	n_0&=n-n_1-n_\star,&\vn&=(n_0,n_\star,n_1),\\
	m_{01}&=m_{10},&m_{00}&=2m-2m_{10}-m_{11},&\vm =(m_{00},m_{01},m_{10},m_{11}).
	\end{align*}
The following lemma provides an asymptotic formula for $|\Gamma_{n,m}(\vN,\vM)|$.

\begin{lemma}\label{Lemma_entropy}
Uniformly in $\vN,\vM$ we have
\begin{align}\label{eq_Lemma_entropy}
|\Gamma_{n,m}(\vN,\vM)|&\sim
\frac{\zeta\exp(dn)\eta(\vn)\kappa(\vm)u(\vn,\vm)}{\Lambda(\vm)}&&&\mbox{where}\\
	\eta(\vn)&=\bink n{\vn}\nu_0^{n_0}\nu_\star^{n_\star}\nu_1^{n_1},&\kappa(\vm)&=(m_{00}-1)!!(m_{11}-1)!!m_{01}!\enspace,\nonumber\\
	\Lambda(\vm)&=\lambda_{00}^{m_{00}}\lambda_{01}^{m_{01}}\lambda_{10}^{m_{10}}\lambda_{11}^{m_{11}},&
		u(\vn,\vm)&=\pr\brk{\hvm=\vm|\hvn=\vn}.\nonumber
	\end{align}
\end{lemma}
\begin{proof}
For a sequence $\vd=(d_{ab}(v))_{v\in[n],a,b\in\{0,1\}}$ let
	\begin{align*}
		\cN_0(\vd)&=\{v\in[n]:d_{10}(v)\leq k-2,d_{01}(v)=d_{11}(v)=0\},\\
	\cN_\star(\vd)&=\{v\in[n]:d_{10}(v)=k-1,d_{00}(v)=d_{11}(v)=0\},\\
	\cN_1(\vd)&=\{v\in[n]:d_{11}(v)\geq k,d_{00}(v)=d_{10}(v)=0\}.
	\end{align*}
Let $\cD(\vn,\vm)$ be the set of all $\vd$ such that $|\cN_0(\vd)|=n_0$, $|\cN_\star(\vd)|=n_\star$,  $|\cN_1(\vd)|=n_1$
and $\sum_{v\in [n]}d_{ab}(v)=m_{ab}$ for all $a,b\in\{0,1\}$.
In addition, let $\cD_0(\vn,\vm)$ be the set of all $\vd$ such that $\cN_0(\vd)=\{1,\ldots,n_0\}$, $\cN_\star(\vd)=\{n_0+1,\ldots,n_0+n_\star\}$
and $\cN_1(\vd)=[n]\setminus(\cN_0(\vd)\cup\cN_\star(\vd))$.
Further, let $s(\vd)$ be the probability that the random graph $\hat\G$ constructed in step (5) 
of {\tt Forge} is simple and that $\hat\vmu=\mu(\hat\G)$.
We claim that
	\begin{align}\label{eqLemma_entropy_1}
	|\Gamma_{n,m}(\vN,\vM)|
	&=\sum_{\vd\in\cD(\vn,\vm)}
	\frac{\kappa(\vm)s(\vd)}{\prod_{v\in[n],a,b\in\{0,1\}}d_{ab}(v)!}
	=\bink{n}{\vn}\sum_{\vd\in\cD_0(\vn,\vm)}
	\frac{\kappa(\vm)s(\vd)}{\prod_{v,a,b}d_{ab}(v)!}.
	\end{align}
Indeed, by \Prop~\ref{Prop_uniform} $|\Gamma_{n,m}(\vN,\vM)|$ is equal to the number of graphs $\hat\G$ that {\tt Forge} can create given the event $\cF(\vN,\vM)$.
Step (2) of {\tt Forge} ensures that given $\cF(\vN,\vM)$ the sequence $\hat{\vec d}=(\hat d_{ab}(v))_{v\in[n],a,b\in\{0,1\}}$ belongs to the set $\cD(\vn,\vm)$.
Furthermore, given $\hat\vd$ the number of possible matchings that step (4) can create is equal to $\kappa(\vm)$, and every possible simple graph can be obtained from exactly $\prod_{v,a,b}d_{ab}(v)!$ matchings.
Thus, we obtain (\ref{eqLemma_entropy_1}).

Proceeding from (\ref{eqLemma_entropy_1}) and observing that $\sum_{a,b\in\{0,1\}}\lambda_{ab}=1$ by the definition (\ref{eqlambda}) of the $\lambda_{ab}$, we obtain
	\begin{align}\label{eqLemma_entropy_2}
	|\Gamma_{n,m}(\vN,\vM)|
	&=\frac{\exp(dn)\kappa(\vm)}{\Lambda(\vm)}\bink{n}{\vn}\sum_{\vd\in\cD_0(\vn,\vm)}s(\vd)
		\prod_{v,a,b}\pr\brk{\Po(\lambda_{ab})=d_{ab}(v)}.
	\end{align}
The definition of $p=p(d,k)$ as the largest fixed point of $\phi_{d,k}$ from (\ref{eqmain}) and the definition (\ref{eq:q}) of $q$ ensure that
	$$\pr\brk{\Po(\lambda_{10})\leq k-2}=1-p,\quad
	\pr\brk{\Po(\lambda_{10})=k-1}=pq,\quad\pr\brk{\Po(\lambda_{11})\geq k}=p(1-q).$$
Therefore, 
letting $\cV=\{\hat\cN_0=[n_0],\hat\cN_1=[n]\setminus[n_\star]\}$,
we can rewrite the product on the right hand side of (\ref{eqLemma_entropy_2}) in terms of the random variables $\hat d_{ab}(v)$ from step (2) of {\tt Forge} as
	\begin{align}
	\prod_{v,a,b}\pr\brk{\Po(\lambda_{ab})=d_{ab}(v)}&=
		{\prod_{1\leq v\leq n_0}
			\pr\brk{\hat d_{00}(v)=d_{00}(v)|\cV}\pr\brk{\hat d_{10}(v)=d_{10}(v)|\cV}\pr\brk{\Po(\lambda_{10})\leq k-2}}\nonumber\\
				&\quad\cdot\prod_{n_0< v\leq n_0+n_\star}
					\pr\brk{\hat d_{01}(v)=d_{01}(v)|\cV}
						\pr\brk{\hat d_{10}(v)=d_{10}(v)|\cV}\pr\brk{\Po(\lambda_{10})=k-1}\nonumber\\
		&\quad\cdot\prod_{n_0+n_\star<v\leq n}\pr\brk{\hat d_{01}(v)=d_{01}(v)|\cV}
				\pr\brk{\hat d_{11}(v)=d_{11}(v)|\cV}\pr\brk{\Po(\lambda_{11})\geq k}
			\nonumber\\
		&=(1-p)^{n_0}(pq)^{n_\star}(p(1-q))^{n_1}\hspace{-1mm}\prod_{v,a,b}\pr\brk{\Po(\hat d_{ab}(v))=d_{ab}(v)|
			\cV}.\label{eqLemma_entropy_3}
	\end{align}
Hence, remembering the definition of $\nu_0,\nu_\star\nu_1$ from (\ref{eq_bar}) and plugging (\ref{eqLemma_entropy_3}) into (\ref{eqLemma_entropy_2}),
we obtain
	\begin{align}\label{eqLemma_entropy_5}
	|\Gamma_{n,m}(\vN,\vM)|
	&=\frac{\eta(\vn)\kappa(\vm)\exp(dn)}{\Lambda(\vm)}
	\sum_{\vd\in\cD_0(\vn,\vm)}s(\vd)
	\prod_{v,a,b}\pr\brk{\hat d_{ab}(v)=d_{ab}(v)|\cV}.
	\end{align}
Moreover, by symmetry with respect to vertex permutations and by \Prop~\ref{Prop_success},
	\begin{align}
	\sum_{\vd\in\cD_0(\vn,\vm)}s(\vd)
		\prod_{v,a,b}\pr\brk{\hat d_{ab}(v)=d_{ab}(v)|\cV}&=
		\Erw[s(\vec{\hat{d}})|\hvm=\vm,\hvn=\vn]\pr\brk{\hvm=\vm|\hvn=\vn}\sim \zeta\pr\brk{\hvm=\vm|\hvn=\vn}.
			\label{eqLemma_entropy_6}
	\end{align}
Finally, the assertion follows from (\ref{eqLemma_entropy_5}) and (\ref{eqLemma_entropy_6}).
\end{proof}

\noindent
As a next step we use Stirling's formula to bring the expression from (\ref{eq_Lemma_entropy}) into a more manageable form.

\begin{corollary}\label{Cor_master}
Uniformly in $\vN,\vM$,
	\begin{align}\label{eq_Cor_master}
		|\Gamma_{n,m}(\vN,\vM)|&\sim
		\frac{\sqrt 2d\zeta u(\vn,\vm)}{\sqrt{pq(1-q)}}\exp\brk{-n\br{\KL{n^{-1}\vn}{\vnu}-\frac d2\KL{(dn)^{-1}\vm}{\vmu}}
		+\frac d2+\frac{d^2}4}\bink{\bink n2}{m}.
	\end{align}	
\end{corollary}
\begin{proof}
Let us begin by approximating the very last factor.
Invoking Stirling's formula, we find
\begin{align}\label{stirl11}
\bink{\bink n2}{m}
&\sim \sqrt{\frac{\bink n2}{2\pi m\br{\bink n2-m}}}
\br{\frac{n(n-1)}{2m}}^m\br{1+\frac{m}{\bink n2-m}}^{\bink n2-m}.
\end{align}
Since $m=\lceil dn/2 \rceil$ we obtain
\begin{align}\label{stirl12}
\br{\frac{n(n-1)}{2m}}^m &\sim \br{\frac{n^2}{2m}}^m \exp\br{-\frac d2}.
\end{align}
Further, the approximation $\ln(1+x)=x-\frac 12x^2 +O(x^3)$ shows that
\begin{align}\label{stirl13}
\br{1+\frac{m}{\bink n2-m}}^{\bink n2-m}
&\sim\exp\br{m-\frac{d^2}4}.
\end{align}
Plugging \eqref{stirl12} and \eqref{stirl13} into
\eqref{stirl11} we obtain
\begin{align}\label{stirl1}
\bink{\bink{n}2}{m}
&
\sim (2\pi m)^{-1/2}\br{\frac{n\eul}{d}}^{m}
\exp\br{-\frac d2-\frac{d^2}4}.
\end{align}
One more application of Stirling's formula and the fact that 
$m=\lceil dn/2 \rceil$  yield
	\begin{align}\label{stirl}
	\sqrt{(2m)!}&\sim \sqrt{2}(\pi m)^{1/4}\br{\frac{dn}{\eul}}^{m}.
	\end{align}
Moreover, combining \eqref{stirl} and \eqref{stirl1} we find
\begin{align}\label{eq_entr2}
\frac{\sqrt{(2m)!}}{d^{dn}}\bink{\bink{n}2}{m}^{-1}
\sim 2(\pi m)^{3/4}\exp\br{\frac d2+\frac{d^2}4}
\exp(-dn).
\end{align}

We proceed to expand $|\Gamma_{n,m}(\vN,\vM)|$ asymptotically.
Let $H$ denote the entropy function defined in 
\eqref{eqHDef}.
By Stirling's formula, our assumption on $\vN$ and the definitions (\ref{eq_bar}) of $\nu_0,\nu_\star,\nu_1$, 
	\begin{align*}
	\bink n{\vn}&\sim(2\pi)^{-1}\sqrt{\frac{n}{n_0n_\star n_1}}\exp(n H(n^{-1}\vn))\sim\frac{\exp(n H(n^{-1}\vn))}{2\pi n\sqrt{\nu_0\nu_1\nu_\star}}
		\sim\frac{\exp\bc{-nH(n^{-1}\vn)}}{2\pi n\sqrt{p^2q(1-p)(1-q)}}.
	\end{align*}
Hence,
	\begin{align}\label{KLn}
	\eta(\vn)
	&\sim\frac{\exp\bc{-n\KL{n^{-1}\vn}{\vnu}}}{2\pi n
	\sqrt{p^2q(1-p)(1-q)}}.
	\end{align}

Further,  (\ref{eq!!}) and Stirling's formula yield
	$$\frac{(m_{ab}-1)!!}{\sqrt{ m_{ab}!}} = \br{2/(\pi m_{ab})}^{1/4}\br{1+O\br{n^{-1}}}\qquad\mbox{for all }a,b\in\{0,1\}.$$
Thus, by (\ref{eq_bar}) and the assumption on $\vM$
	\begin{align}\nonumber
	\kappa(\vm)&=(m_{00}-1)!!(m_{11}-1)!!m_{01}!\sim\sqrt\frac{2}{\pi}\cdot\sqrt{m_{00}!m_{01}!m_{10}!m_{11}!}\cdot(m_{00}m_{11})^{-1/4}\\
		&\sim\sqrt\frac{2}{\pi p(1-p)m}\cdot\sqrt{m_{00}!m_{01}!m_{10}!m_{11}!}\enspace.\label{eqKLn2}
	\end{align}
Since $\Lambda(\vm)=d^{dn}\prod_{a,b}\mu_{ab}^{m_{ab}/2}$, the definition (\ref{eq_bar}) of the $\mu_{ab}$ and (\ref{eqKLn2}) yield
\begin{align}
\frac{\kappa(\vm)}{\Lambda(\vm)}
&\sim\frac{\sqrt{(2m)!}}{d^{dn}\sqrt{\pi m p(1-p)}} 
\bink{2m}{\vm}^{-1/2}\prod_{a,b}\mu_{ab}^{-m_{ab}/2}.
\label{eq_entr1}
\end{align}	
Further, applying Stirling's formula and using the assumption on $\vM$, we obtain
\begin{align}\label{KLm}
\bink{2m}{\vm}\prod_{a,b}\mu_{ab}^{m_{ab}}
&\sim\frac{ \exp\br{-2m\KL{(2m)^{-1}\vm}{\vmu}}}{(4\pi m)^{3/2} p^2(1-p)^2}.
\end{align}
Thus, combining  \eqref{eq_entr2}, \eqref{eq_entr1} and \eqref{KLm}, we obtain
\begin{align}\label{kappalambda}
\frac{\kappa(\vm)}{\Lambda(\vm)}
&\sim 2^{5/2}\pi m \sqrt{p(1-p)} \exp\br{-dn + m \KL{(2m)^{-1}\vm}{\vmu}
+\frac d2+\frac{d^2}4}
\bink{\bink n2}{m}.
\end{align}
Plugging in \eqref{kappalambda} and  \eqref{KLn} into \eqref{eq_Lemma_entropy} completes the proof.
\end{proof}

\Cor~\ref{Cor_master} provides an explicit formula for $|\Gamma_{n,m}(\vN,\vM)|$, apart from the conditional probability $u(\vn,\vm)=\pr\brk{\hvm=\vm|\hvn=\vn}.$
As a next step we will derive an explicit expression for  $u(\vn,\vm)$.
To this end we introduce the matrices
\begin{align}\label{eqSigmaMatrix}
	\Sigma
	&=\frac1d
	\begin{pmatrix}
	(1-p)^2&0&0&0\\
	0&p(1-p)&0&0\\
	0&0&p(1-p)\bc{1+\bar q\br{dp(1-\bar q)-(k-1)}}&0\\
  0&0&0&  p^2\brk{1-\frac {dp}{1-q}+d(p+(1-p)\bar q)}
	\end{pmatrix}
\end{align}
and
\begin{align}\label{eqLMatrix}
	L&
	=\begin{pmatrix}
	1-p & 0 & 0\\
	0 & 1-p & 1-p\\
	p(1-\bar q) & (k-1)/d & 0\\
	0 & 0 & p/(1-q)\end{pmatrix}.
\end{align}

\begin{lemma}\label{Prop_degrees}
Let $k\geq 3, d>d_k$ and let $\xi>0$.
{Then $\Sigma$ is regular.}
Moreover, let $\vn=(n_0,n_\star, n_1)$ be such that $n_0+n_\star+ n_1=n$ and $|n_\star -n\nu_\star|+|n_1-n\nu_1|\leq \xi\sqrt n$.
Then uniformly for all $\vm \in\mathbb N^4$,
\begin{align*}
	u(\vn,\vm)
    &=\frac{1}
   {(2\pi n)^2d^4\sqrt{\det{\Sigma}}}
    \exp\br{-\frac{n}{2}
    \bck{
	\begin{pmatrix}  L^*\Sigma^{-1}L & 
	-L^*\Sigma^{-1}\\-\Sigma^{-1}L &
	\Sigma^{-1} \end{pmatrix}
   \bink{\Delta(\vn)}{\Delta(\vm)}, \bink{\Delta(\vn)}{\Delta(\vm)}}}+\mbox{o}(n^{-2})
\end{align*}
where
\begin{align}\label{eq_Prop_degrees_DELTA}
\Delta(\vn)
&=\br{\frac{n_0}{n}-\nu_0,\frac{n_\star}{n}-\nu_\star,\frac{n_1}{n}-\nu_1}\trans,&
\Delta(\vm)&=\br{
\frac{m_{00}}{2m}-\mu_{00},\frac{m_{01}}{2m}-\mu_{01},
\frac{m_{10}}{2m}-\mu_{10},\frac{m_{11}}{2m}-\mu_{11}}\trans.
\end{align}
\end{lemma}
\begin{proof}
Given $\hat\cN_0,\hat\cN_\star,\hat\cN_1$, we can characterise the distributions of the random variables $\hat d_{a,b}(v)$ from step (1) of {\tt Forge} as follows in terms of the $\lambda_{00},\ldots,\lambda_{11}$ from (\ref{eqlambda}):
	\begin{align*}
	\hat d_{00}(v)&\stacksign d=\Po(\lambda_{00}),
		&\hat d_{01}(v)&=0,&\hat d_{10}(v)&\stacksign d=\Po_{\leq k-2}(\lambda_{10}),
		&\hat d_{11}(v)&=0&\mbox{given }v\in\hat\cN_0,\\
	\hat d_{00}(v)&=0,
		&\hat d_{01}(v)&\stacksign d=\Po(\lambda_{01}),&\hat d_{10}(v)&=k-1,
		&\hat d_{11}(v)&\stacksign d=0&\mbox{given }v\in\hat\cN_\star,\\	
	\hat d_{00}(v)&=0,
		&\hat d_{01}(v)&\stacksign d=\Po(\lambda_{01}),&\hat d_{10}(v)&=0,
		&\hat d_{11}(v)&\stacksign d=\Po_{\geq k}(\lambda_{11})&\mbox{given }v\in\hat\cN_1.
	\end{align*}
Hence, for an arbitrary $v\in[n]$  and $x\in\{0,\star,1\}$ let
	$$\hat{\vec a}_x=(\Erw[\hat d_{00}(v)|v\in\hat\cN_x],\Erw[\hat d_{01}(v)|v\in\hat\cN_x],\Erw[\hat d_{10}(v)|v\in\hat\cN_x],\Erw[\hat d_{11}(v)|v\in\hat\cN_x])\trans$$
and $\hat{\vec{a}}=\sum_{x\in\{0,\star,1\}}\frac{\hat n_x}n\hat{\vec{a}}_x$.
Further, let
	$$\hat D_x=\begin{pmatrix}
			\Var[\hat d_{00}(v)|v\in\hat\cN_x]&0&0&0\\
			0&\Var[\hat d_{01}(v)|v\in\hat\cN_x]&0&0\\
			0&0&\Var[\hat d_{10}(v)|v\in\hat\cN_x]&0\\
				0&0&0&\Var[\hat d_{11}(v)|v\in\hat\cN_x]
			\end{pmatrix}$$
and $\hat D =\sum_{x\in\{0,\star,1\}}\nu_x\hat D_x$. 
{By definition of $\hat d_{ab}(v), a,b\in \{0,1\}$
we obtain that $\hat D$ is regular.} 
Further,
because the random variables $(\hat d_{a,b}(v))_{v,a,b}$ are mutually independent,
given $\{\hat{\vn}=\vn\}$ the sequence $n^{-1/2}(\hat{\vm}-n\hat{\vec{a}})$
converges in distribution to a multivariate normal distribution with covariance matrix $\hat D$ and mean
$(0,0,0,0)$.
Indeed, \Thm~\ref{Thm_McD} implies that uniformly for all $\vec m \in\mathbb N^4$,
\begin{align}\label{McD}
 \pr\brk{\hat{\vm}=\vm|\hat{\vn}=\vn} 
 &= \frac{
 \exp\br{-\frac n2 \bck{{\hat D}^{-1}\br{\vm/n
 -\hat{\vec{a}}},
 \br{\vm/n-\hat{\vec{a}}}}}}{(2\pi n)^2\sqrt{\det \hat D}}
 + \mbox{o}\br{n^{-2}}.
\end{align}

Hence, to complete the proof we just need to calculate $\hat{\vec a}$ and $\hat D$ explicitly.
We claim that
\begin{align}\label{eqhatveca}
 \hat{\vec a}_0 
 =
 \begin{pmatrix}
 d(1-p) \\0\\ dp(1-\bar q)\\ 0
 \end{pmatrix},
\qquad
 \hat{\vec a}_\star 
 =
 \begin{pmatrix}
 0\\ d(1-p)\\ k-1\\ 0
 \end{pmatrix},
\qquad 
 \hat{\vec a}_1 
 =
 \begin{pmatrix}
 0\\ d(1-p)\\ 0\\ dp/(1-q)
 \end{pmatrix}.
\end{align}
Indeed, remembering (\ref{eqlambda}), we see that
\begin{align}\label{McD_E1}
 \Erw\brk{\Po(\lambda_{00})} &= \lambda_{00}=d(1-p), 
 &\Erw\brk{\Po(\lambda_{01})} &= \lambda_{01}=d(1-p).
\end{align} 
Furthermore, remembering (\ref{eq:q}) and (\ref{eq:qbar}),
\begin{align}
 \Erw\brk{\Po_{\leq k-2}(\lambda_{10})}
 =\frac{1}{1-p}\sum_{i\leq k-2} \frac{i(dp)^i}{i!\exp(dp)}
 =\frac{dp}{1-p}\pr\brk{\Po(dp) \leq k-3}
 =dp(1-\bar q),\label{McD_E2b}\\
 \Erw\brk{\Po_{\geq k}(\lambda_{11})}
 =\frac{1}{p(1-q)}\sum_{i\geq k} \frac{i(dp)^i}{i!\exp(dp)}
 =\frac{dp}{p(1-q)}\pr\brk{\Po(dp) \geq k-1}
 =\frac{dp}{1-q}\label{McD_E3a}
\end{align}
and (\ref{eqhatveca}) is immediate from (\ref{McD_E1})--(\ref{McD_E3a}).
Moving on to the covariance matrix $\hat D$,  we clearly have
\begin{align}\label{McD_Var1}
 \Var\brk{\Po(\lambda_{00})} &= \lambda_{00}=d(1-p), 
 &\Var\brk{\Po(\lambda_{01})} &= \lambda_{01}=d(1-p).
\end{align} 
Moreover, by the definition (\ref{eq:qbar}) of $\bar q$,
	\begin{equation}\label{McD_Var2}
	\pr\brk{\Po(dp)=k-2}=(1-p)\bar q.
	\end{equation}
Furthermore, 
	\begin{equation}\label{McD_Var3}
	\pr\brk{\Po(dp)=k-3}=\frac{k-2}{dp}\pr\brk{\Po(dp)=k-2}.
	\end{equation}
Hence, using 	\eqref{McD_Var2} we obtain 
\begin{align}\nonumber
 \Erw\brk{\hat d_{10}(v)(\hat d_{10}(v)-1)|v\in\hat\cN_0}
& = \frac{1}{1-p}\sum_{i\leq k-2} i(i-1)\frac{(dp)^i}{i!\exp(dp)}
  = \frac{(dp)^2}{1-p}\pr\brk{\Po(dp)\leq k-4}\\
 & = \frac{(dp)^2}{1-p}\br{1-p-(1-p)\bar q-(1-p)\bar q\frac{k-2}{pd}}
  = (dp)^2\br{1-\bar q-\bar q\frac{k-2}{pd}}.\label{McD_Var4}
\end{align}
Similarly, by \eqref{McD_Var3} 
\begin{align}\label{McD_Var5}
{\Erw\brk{\hat d_{11}(v)(\hat d_{11}(v)-1)|v\in\hat\cN_1}}
& =\frac{(dp)^2}{p(1-q)}\pr\brk{\Po(dp)\geq k-2}
 =\frac{d^2p}{1-q}\br{p+(1-p)\bar q}.
\end{align}
Combining \eqref{McD_E2b} and \eqref{McD_Var4} as well as \eqref{McD_E3a} 
and \eqref{McD_Var5} and using that $\Var(X)=\Erw(X)-\Erw(X)^2+ \Erw(X(X-1))$, we obtain
\begin{align}\label{McD_Var6a}
\Var (\hat d_{10}(v)|v\in\hat\cN_0)
& = dp(1-\bar q(k-1)) + (dp)^2\bar q (1+\bar q),\\
\Var(\hat d_{11}(v)|v\in\hat\cN_1)&=\frac{dp}{1-q}-\br{\frac{dp}{1-q}}^2+pd^2\frac{p+(1-p)\bar q}{1-q}.\label{McD_Var6b}
\end{align}
Combining \eqref{McD_Var1}, (\ref{McD_Var6a}) and \eqref{McD_Var6b}, we obtain
\begin{align*}
 \hat D_0&=\begin{pmatrix}d(1-p)&0&0&0\\0&0&0&0\\0&0&dp(1-\bar q(k-1)) + (dp)^2\bar q (1-\bar q)&0\\0&0&0&0
 	\end{pmatrix},&
 \hat D_\star&=\begin{pmatrix}
 	0&0&0&0\\0&d(1-p)&0&0\\0&0&0&0\\0&0&0&0
	\end{pmatrix},\\
 \hat D_1&=\begin{pmatrix}
 	0&0&0&0\\
	0&d(1-p)&0&0\\
	0&0&0&0\\
	0&0&0&\frac{dp}{1-q}-\br{\frac{dp}{1-q}}^2+pd^2\frac{p+(1-p)\bar q}{1-q}
 	\end{pmatrix}.
\end{align*}
Finally, we verify that the matrices $\Sigma,L$ from (\ref{eqSigmaMatrix}) and (\ref{eqLMatrix}) satisfy
$\vm/n-\hat{\vec{a}} 
=d\br{\Delta(\vm)-L\Delta(\vn)}$
and $\hat D=\sum_x\nu_x\hat D_x=d^2\Sigma.$
{Since $\hat D$ is regular, we obtain that $\Sigma$
is regular.}
Hence,
\begin{equation}\label{McD_matrix}
\bck{{\hat D}^{-1}\br{\vm/n
 -\hat{\vec{a}}},
 \br{\vm/n-\hat{\vec{a}}}}
 =
     \bck{
	\begin{pmatrix}  L^*\Sigma^{-1}L & 
	-L^*\Sigma^{-1}\\-\Sigma^{-1}L &
	\Sigma^{-1} \end{pmatrix}
\begin{pmatrix}   
   \Delta(\vn)\\
   \Delta(\vm)
   \end{pmatrix},\begin{pmatrix}   
   \Delta(\vn)\\
   \Delta(\vm)
   \end{pmatrix}}.
\end{equation}
Plugging in \eqref{McD_matrix} in \eqref{McD}, we obtain the assertion because $\det\hat D= d^8 \det\Sigma.$
\end{proof}

\begin{proof}[Proof of \Prop~\ref{Prop_entropy}.]
We are going to prove \Prop~\ref{Prop_entropy}
by combining \Cor~\ref{Cor_master}
and \Lem~\ref{Prop_degrees}.
To this end, we remember the Taylor expansion of  the Kullback-Leibler divergence $\KL{\cdot}{\cdot}$ from  \eqref{eq_Cor_master}.
Using (\ref{eqDiffKL}), we see that the first derivative of $\KL{\cdot}{\vnu}$ vanishes at the point $\vnu$, where the global minimum of $0$ is attained, 
and similarly $\KL{\cdot}{\vmu}$ attains its global minimum of $0$ at $\vmu$.
Expanding the Kullback-Leibler divergence to the second order, we obtain with $\Delta(\vn)$, $\Delta(\vm)$ from (\ref{eq_Prop_degrees_DELTA}) that
\begin{align}\label{eq_Taynu}
 \KL{n^{-1}\vn}{\vnu}
 = \frac 12 \bck{\diag(\vnu)^{-1}\Delta(\vn),\Delta(\vn)}
 +O\br{n^{-3/2} },\\
\label{eq_Taymu}
\KL{(dn)^{-1}\vm}{\vmu}
 = \frac 12 \bck{\diag(\vmu)^{-1}\Delta(\vm),\Delta(\vm)}
 +O\br{n^{-3/2} }.
\end{align}
Further,
\begin{align}\label{eq_scalar}
    &\bck{
	\begin{pmatrix}  L^*\Sigma^{-1}L & 
	-L^*\Sigma^{-1}\\-\Sigma^{-1}L &
	\Sigma^{-1} \end{pmatrix}
   \bink{\Delta(\vn)}{\Delta(\vm)}, \bink{\Delta(\vn)}{\Delta(\vm)}}
   +\bck{\diag(\vnu)^{-1}\Delta(\vn),\Delta(\vn)}
   -\frac d2 \bck{\diag(\vmu)^{-1}\Delta(\vm),\Delta(\vm)} \nonumber \\
   &\quad\quad=
   \bck{
	\begin{pmatrix}  L^*\Sigma^{-1}L+\diag(\vnu)^{-1}& -L^*\Sigma^{-1}\\-\Sigma^{-1}L & \Sigma^{-1}-\frac d2 \diag(\vmu)^{-1}\end{pmatrix}
 \bink{\Delta(\vn)}{\Delta(\vm)},\bink{\Delta(\vn)}{\Delta(\vm)}}.
\end{align}
Combining (\ref{eq_Taynu}), (\ref{eq_Taymu}) and (\ref{eq_scalar})  {with \Cor~\ref{Cor_master} and \Lem~\ref{Prop_degrees}}, we obtain
\begin{align}\label{eq_matrixfinal}
\frac{ |\Gamma_{n,m}(\vN,\vM)|}{ \bink{\bink n2}{m}} &\sim
 \frac{C}{n^2}\cdot
 \exp\br{-\frac n2 \bck{
 \begin{pmatrix}  L^*\Sigma^{-1}L+\diag(\vnu)^{-1}& -L^*\Sigma^{-1}\\-\Sigma^{-1}L & \Sigma^{-1}-\frac d2 \diag(\vmu)^{-1}\end{pmatrix}
 \bink{\Delta(\vn)}{\Delta(\vm)},\bink{\Delta(\vn)}{\Delta(\vm)}}},\quad\mbox{with}\\
 C&=\frac{\sqrt{2}\zeta}{(2\pi)^2d^3\sqrt{pq(1-q)\det\Sigma}}\enspace.\nonumber
\end{align}
To proceed, let
	\begin{align*}
	T
	&=\begin{pmatrix*}[r]
	-1 & -1 & 0 & 0\\
	1 & 0 & 0 & 0\\
	0 & 1 & 0 & 0\\
	0 & 0 & -2 & -1\\
	0 & 0 & 1 & 0\\
	0 & 0 & 1 & 0
	\\0 & 0 & 0 & 1\\
	\end{pmatrix*}.
	\end{align*}
Then the vector $\bink{\Delta(\vn)}{\Delta(\vm)}$ can be written as $T\Delta(\vN,\vM)$,  with $\Delta(\vN,\vM)$ from (\ref{eqDELTANM}).
{
By means of a computer algebra system
\footnote{We use 
the free open-source mathematics software system SageMath.
An  executable code file and PDF version of the 
source code
are provided at 
{\tt http://www.uni-frankfurt.de/53778787}. 
SageMath worksheets can be executed using the online platform
CoCalc, see {\tt https://cocalc.com/}.} 
we verify that 
	$$C=\frac{1}{2\pi^2d^2\sqrt{\det(Q)}}.$$
Using \Lem~\ref{Prop_degrees} this implies that $Q$ is a regular matrix.
Finally,
 calculating the entries of the matrix on the right hand side explicitly (for which once more we use a computer algebra system), 
we see that the matrix $Q$ from (\ref{eqQmatrix}) satisfies
	\begin{align*}
Q^{-1}
	&=T^*
	\begin{bmatrix}  L^*\Sigma^{-1}L +\diag(\vnu)^{-1} & -L^*
	\Sigma^{-1}\\-\Sigma^{-1}L & \Sigma^{-1} - \frac d2 
	\diag(\vmu)^{-1}\end{bmatrix}T.
	\end{align*}}
Hence, (\ref{eq_matrixfinal}) can be written as
\begin{align*}
\frac{ |\Gamma_{n,m}(\vN,\vM)|}{ \bink{\bink n2}{m}} &\sim
 \frac{1}{2\pi^2d^2n^2\sqrt{\det(Q)}}
 \exp\br{-\frac n2 \bck{
 Q^{-1}\Delta(\vN,\vM),\Delta(\vN,\vM)}},
\end{align*}
as desired.
\end{proof}

\end{document}